\documentclass[11pt]{amsart} 
\usepackage{amscd,amssymb,amsxtra}
\usepackage{mathrsfs}
\usepackage{mathabx}
\usepackage{comment}
\usepackage{booktabs}
\makeatletter
 \def\section{\@startsection{section}{1}%
  \z@{4\linespacing\@plus\linespacing}{\linespacing}%
  {\bfseries\centering}}
 \def\subsection{\@startsection{subsection}{2}%
   \z@{1.25\linespacing\@plus.7\linespacing}{.5\linespacing}%
   {\normalfont\bfseries}}
\makeatother

\theoremstyle{definition}
\newtheorem{example}[equation]{Example}

\newtheorem*{example*}{Example}
\newtheorem*{problem*}{Problem}
\newtheorem*{exercise*}{Exercise}
\newtheorem*{question*}{Question}

\theoremstyle{remark}

\newtheorem{notation}[equation]{Notation}

\newtheorem{remark}[equation]{Remark}

\newtheorem*{note*}{Note}
\newtheorem*{notation*}{Notation}
\newtheorem*{claim*}{Claim}
\newtheorem*{remark*}{Remark}

\theoremstyle{plain} \newtheorem{definition}[equation]{Definition}
\newtheorem{theorem}[equation]{Theorem}

\newtheorem{lemma}[equation]{Lemma}
\newtheorem{proposition}[equation]{Proposition}

\newtheorem*{definition*}{Definition}
\newtheorem*{theorem*}{Theorem}
\newtheorem*{corollary*}{Corollary}
\newtheorem*{lemma*}{Lemma}
\newtheorem*{proposition*}{Proposition}
\newtheorem*{conjecture*}{Conjecture}

\numberwithin{equation}{section}

\renewcommand{\:}{\colon}

\DeclareMathOperator{\Aut}{Aut}
\newcommand{\CC}{{\mathbb C}}

\DeclareMathOperator{\Hom}{Hom}
\DeclareMathOperator{\id}{id}
\DeclareMathOperator{\Map}{Map}

\DeclareMathOperator{\pt}{pt}
\newcommand{\QQ}{{\mathbb Q}}

\newcommand{\RR}{{\mathbb R}}
\newcommand{\TT}{\mathbb T}
\DeclareMathOperator{\Spin}{Spin}

\DeclareMathOperator{\tr}{tr}

\newcommand{\ZZ}{{\mathbb Z}}

\newcommand{\chiup}{\raise.5ex\hbox{$\chi$}}
\newcommand{\cir}{S^1}

\newcommand{\inv}{^{-1}}
\newcommand{\mstrut}{^{\vphantom{1*\prime y}}}

\newcommand{\temsquare}{\raise3.5pt\hbox{\boxed{ }}}

\newcommand{\zmod}[1]{\ZZ/#1\ZZ}

\newcommand{\zt}{\zmod2}

\usepackage[all,2cell,dvips]{xy}\renewcommand{\cir}{\ensuremath{S^1}}
\usepackage{enumerate}

\DeclareMathOperator{\1-mor}{1-mor}
\DeclareMathOperator{\2-mor}{2-mor}
\DeclareMathOperator{\Bun}{Bun}
\DeclareMathOperator{\Det}{Det}
\DeclareMathOperator{\Euler}{Euler}
\DeclareMathOperator{\Fam}{Fam}
\DeclareMathOperator{\OO}{O}
\DeclareMathOperator{\SO}{SO}

\DeclareMathOperator{\Sum}{Sum}
\DeclareMathOperator{\Sym}{Sym}
\DeclareMathOperator{\Vect}{Vect}
\DeclareMathOperator{\fr}{fr}
\DeclareMathOperator{\ob}{ob}
\DeclareMathOperator{\sign}{sign}
\DeclareMathOperator{\vol}{vol}
\newcommand{\Alg}{\mathcal{A}lg}
\newcommand{\Aop}{A^{\textnormal{op}}}
\newcommand{\At}{\A\mstrut _{\mathfrak{t}}}
\newcommand{\A}{\mathscr{A}}
\newcommand{\cA}{\mathscr{A}}
 
\newcommand{\BordSO}{\mathbf{Bord}^{SO}} 
\newcommand{\Bord}{\mathbf{Bord}}
\newcommand{\BunG}{\Bun_G}
\newcommand{\B}{\mathcal{B}}
\newcommand{\CG}{\CC [G]}
\newcommand{\C}{\mathcal{C}}
\newcommand{\D}{\mathcal{D}}
\newcommand{\Ft}{\F_{\mathfrak{t}}}
\newcommand{\F}{\mathcal{F}}
\newcommand{\G}{\mathcal{G}}
\newcommand{\QZ}{\QQ/\ZZ}
\newcommand{\RZ}{\RR/\ZZ}

\newcommand{\VectC}{\Vect\mstrut _{\CC}}
\newcommand{\bimod}[2]{(#1,#2)}
 
\newcommand{\bordSOG}{\textnormal{Bord}^{SO}_2(G)} 
\newcommand{\bordSO}{\textnormal{Bord}^{SO}} 
 
\newcommand{\cHP}{\widecheck{H\Pi }}
\newcommand{\dA}{A^\vee}
\newcommand{\famn}{\Fam_n}
\newcommand{\form}{b}
\newcommand{\frt}{\mathfrak{t}}
\newcommand{\gpd}{/\!/}
\newcommand{\hF}{\hat{F}}
\newcommand{\hT}{\hat{T}}
\newcommand{\qQZ}{q\mstrut _{\QZ}}
\newcommand{\tCG}{\CC ^\tau [G]}
\newcommand{\tF}{F^\tau}
\newcommand{\tG}{G^\tau }
\newcommand{\tK}{K^\tau }
\newcommand{\tRi}{\tau \inv _{\RR}}
\newcommand{\tR}{\tau \mstrut _{\RR}}

\newcommand{\tVect}{\Vect^\tau}
\newcommand{\tqQZ}{\tilde{q}\mstrut _{\QZ}}
\newcommand{\tq}{\tilde{q}}
\newcommand{\ttheta}{\theta ^\tau }

\newcommand{\cB}{\mathscr{B}}
\newcommand{\bC}{\mathbb{C}}

\newcommand{\bT}{\mathbb{T}}
\newcommand{\bZ}{\mathbb{Z}}
\newcommand{\sky}{\mathrm{Sky}}		
\newcommand{\vect}{\mathrm{Vect}}
\newcommand{\pic}{\mathrm{Pic}}

\begin{document}

\abovedisplayskip18pt plus4.5pt minus9pt
\belowdisplayskip \abovedisplayskip
\abovedisplayshortskip0pt plus4.5pt
\belowdisplayshortskip10.5pt plus4.5pt minus6pt
\baselineskip=15 truept
\marginparwidth=55pt




 \title[TQFTs from compact Lie groups]{Topological Quantum Field Theories
from Compact Lie Groups} 

 \author[D. S. Freed]{Daniel S.~Freed}
 \thanks{The work of D.S.F. is supported by NSF grant DMS-0603964}
 \address{Department of Mathematics \\ University of Texas \\ 1 University
Station C1200\\ Austin, TX 78712-0257}
 \email{dafr@math.utexas.edu}

  \author[M. J. Hopkins]{Michael J. Hopkins} 
  \thanks{The work of M.J.H. is supported by NSF grants DMS-0306519 and
DMS-0757293 as well as the DARPA grant FA9550-07-1-0555.}
  \address{Department of Mathematics\\Harvard University\\ 1 Oxford Street \\
 Cambridge, MA 02138}
  \email{mjh@math.harvard.edu}

\author{Jacob Lurie}
\address{Department of mathematics\\Massachusetts Institute of Technology\\
77 Massachusetts Avenue\\Cambridge, MA 02139-4307}
\email{lurie@math.mit.edu}
\thanks{The work of J.L. is supported by the American Institute of
Mathematics.}
 
 \author[C. Teleman]{Constantin Teleman} 
  \thanks{The work of C.T. is supported by NSF grant DMS-0709448}
 \address{Department of Mathematics \\ University of California \\ 970 Evans
Hall \#3840 \\ Berkeley, CA 94720-3840}  
 \email{teleman@math.berkeley.edu}

 \date{June 19, 2009}
\maketitle


Let $G$~be a compact Lie group and $BG$~a classifying space for~$G$.  Then a
class in~$H^{n+1}(BG;{\mathbb Z})$ leads to an $n$-dimensional topological
quantum field theory~(TQFT), at least for~$n=1,2,3$.  The theory for~$n=1$ is
trivial, but we include it for completeness.  The theory for~$n=2$ has some
infinities if $G$~is not a finite group; it is a topological limit of
2-dimensional Yang-Mills theory.  The most direct analog for~$n=3$ is an
$L^2$~version of the topological quantum field theory based on the classical
Chern-Simons invariant, which is only partially defined.  The TQFT
constructed by Witten and Reshetikhin-Turaev which goes by the name
`Chern-Simons theory' (sometimes `\emph{holomorphic} Chern-Simons theory' to
distinguish it from the $L^2$~theory) is completely finite.

The theories we construct here are extended, or multi-tiered, TQFTs which go
all the way down to points.  For the $n=3$~Chern-Simons theory, which we term
a `0-1-2-3~theory' to emphasize the extension down to points, we only treat
the cases where $G$~is finite or $G$~is a torus, the latter being one of the
main novelties in this paper. In other words, for toral theories we provide
an answer to the longstanding question: What does Chern-Simons theory attach
to a point?  The answer is a bit subtle as Chern-Simons is an
\emph{anomalous} field theory of oriented manifolds.\footnote{It is not
anomalous as a theory of framed manifolds.}  This framing anomaly was already
flagged in Witten's seminal paper~\cite{Wi}.  Here we interpret the anomaly
as an \emph{invertible} 4-dimensional topological field theory~$\A$, defined
on oriented manifolds.  The Chern-Simons theory is a ``truncated morphism''
$Z\:1\to \A$ from the trivial theory to the anomaly theory.  For example, on
a closed oriented 3-manifold~$X$ the anomaly theory produces a complex
line~$\A(X)$ and the Chern-Simons invariant~$Z(X)$ is a (possibly zero)
element of that line.  This is the standard vision of an anomalous quantum
field theory in general; here we use this description down to points.  The
invariant of a 4-manifold in the theory~$\A$ involves its signature and Euler
characteristic.  It was first discovered in a combinatorial
description~\cite{CY} and Walker~\cite{W} also uses~$\A$ in his description
of Chern-Simons (for a more general class of gauge groups).

Since a torus is an abelian group, the classical Chern-Simons action is
quadratic in the connection and so the theory is in some sense ``free''.
Indeed, one expects that the semi-classical approximation is exact.  This is
the point of view taken by Manoliu~\cite{Ma}, who constructs Chern-Simons for
circle groups as a 2-3~theory.  The invariant this theory assigns to a closed
oriented 3-manifold is made from classical invariants of 3-manifold topology:
is the integral over the space of flat connections of the square root of the
Reidemeister torsion times a spectral flow phase.  (There is an overall
``volume'' factor as well.)  The extension to a 1-2-3~theory is determined by
a \emph{modular tensor category} by a theorem of Reshetikhin and
Turaev~\cite{T}; it is assigned to the circle by the theory.  For toral
Chern-Simons this is a well-known category, and its relation to Chern-Simons
theory was explored recently by Stirling~\cite{St}, inspired by earlier work
by Belov-Moore~\cite{BM}.  In a different direction the Verlinde ring, which
encodes the 2-dimensional reduction of Chern-Simons theory, was recognized by
three of the authors as a twisted form of equivariant $K$-cohomology~\cite{FHT}.
That description, or rather its equivalent in $K$-homology, inspires the
categories of skyscraper sheaves which we use to extend toral Chern-Simons to
a 0-1-2-3~theory. 

The gauge theories discussed in this paper have a classical description as
pure gauge theories whose only field is a $G$-bundle with connection.  On any
fixed manifold~$M$ the $G$-bundles with connection form a groupoid, and it is
the underlying stack which should be regarded as the fields in the theory.
When $M$ is a point, there is up to isomorphism a unique $G$-bundle with
connection (the trivial one), but it has a group of automorphisms isomorphic
to~$G$.  Thus the stack of $G$-bundles with connections is~$*\gpd G=BG$.  A
cohomology class in~$H^{n+1}(BG; {\mathbb Z})$, or rather the geometric
manifestation which we describe (esp.~\ref{subsec:1.4} and~\ref{subsec:3.1}),
defines the classical gauge theory on~$M \simeq \pt$.  We directly
``quantize'' this classical data over a point.  This may be regarded, as
in~\cite{F2}, as the ``path integral'' over the fields on a point.
Integration over~$I\gpd G$, which is implemented as a categorical limit,
amounts to taking $G$-invariants, which is a higher version of the Gauss Law
in canonical quantization.  It is also worth noting that the space of
connections on~$M=\cir$ is the finite dimensional groupoid~$G\gpd G$, where
$G$~acts on~$G$ by conjugation.  Properly interpreted, this is the loop space
of the fields on a point.  On the other hand, for any manifold of dimension
at least two, the stack of connections is infinite dimensional unless $G$~is
finite.

In the latter case---for a finite group~$G$---the path integral in all
dimensions reduces to a finite sum, so is manifestly well-defined and
satisfies the axioms of a field theory.  This idea was initiated
in~\cite{F2}, where the notion of a classical action in an $n$-dimensional
field theory is extended to manifolds of all dimensions~$<n$ and a higher
categorical version of the path integral is used to heuristically define the
quantum invariants.\footnote{The application in~\cite{F2} is to \emph{derive}
the quantum group which appears in 3-dimensional Chern-Simons theory from the
classical Chern-Simons action.}  The developments in higher category theory
in the intervening years, and the new definitions and structure theory for
TQFT, make it possible to give a rigorous treatment of these finite path
integrals and to generalize them.  We give some indications in~\S\ref{sec:5}
and~\S\ref{cs8}; we hope to develop these ideas in detail elsewhere.
 
That structure theorem is the Baez-Dolan cobordism hypothesis, for which we
include an exposition in~\ref{subsec:2.2}.  A much more detailed account may
be found in~\cite{L}, which also contains a detailed sketch of the proof.
This theorem asserts that a fully extended TQFT---that is, a TQFT extended
down to points---is determined by its value on a point.  Furthermore, it
characterizes the possible values on a point.  This theorem provides one
possible construction of many of the theories in this paper.\footnote{A
corollary of the cobordism hypothesis characterizes theories on oriented and
spin manifolds: for these the value on a point is endowed with certain
equivariance data.  For the 3-~and 4-dimensional theories of most interest we
do not pursue this equivariance data here.}

The anomaly theory~$\A$ which appears in Chern-Simons theory has a classical
description as a 4-dimensional field theory.  In~\ref{subsec:6.1} we sketch
two such classical theories, each based on a 2-groupoid of gerbes (rather
than a 1-groupoid of connections, which is what appears in gauge theory).
These are free theories, so the path integral on a 4-manifold is Gaussian and
can be carried out explicitly.  Presumably this is a classical description of
the theories put forward by Crane and Yetter~\cite{CY}, though we do not
attempt to make the connection.  One of these classical theories is based on
a finite group, and for that we can apply the finite path integrals
of~\S\ref{sec:5} (in the generalized form for 2-groupoids) to construct the
quantum theory.  In~\S\ref{sec:cs6} we start at the other end of the theory,
that is, with points rather than 4-manifolds.  We describe two braided tensor
categories which may be attached to a point in the anomaly theory~$\A$; they
are Morita equivalent, so both define~$\A$.  In~\ref{subsec:c6.2} we
introduce a finite (lattice) approximation of the continuous version of the
anomaly theory, and so obtain many more Morita equivalent braided tensor
categories which define that theory (Proposition~\ref{morita}).  Finally,
in~\ref{cs9.2} we spell out what the anomaly theory assigns to manifolds of
various dimensions from~$0$ to~$4$.
 
The Chern-Simons theory proper is described in~\ref{sec:c9}; see
Theorem~\ref{thm:9.1}.  The tables in~\ref{subsec:c9.3} are helpful in
organizing the motley characters who play a part in our story.  Our account
here emphasizes 0-manifolds as that is the new element.

We begin in \S\ref{sec:1} with a discussion of the 1-dimensional case to set
down some basic notions.  Here the gauge group is an arbitrary compact Lie
group~$G$, and the path integral reduces to an integral over~$G$.  The
2-dimensional case~(\S\ref{sec:2}) exhibits more features.
In~\S\ref{sec:3} we discuss the 3-dimensional
gauge theory with finite gauge group.  On the one hand this is an application
of the finite path integral and on the other a warm-up for the treatment
in later sections of torus groups.  Section~\ref{sec:c1} is a set of variations
on `algebra' of increasing (categorical) complexity.  It provides some
underpinnings for our discussion of the anomaly theory~$\A$ and Chern-Simons
theory. 

As befits a conference proceedings, our presentation here emphasizes the big
picture and our speculation is uninhibited; many details are not worked out.

The recent paper~\cite{BN} generalizes the 3-dimensional untwisted
Chern-Simons theory for finite groups in a different direction and gives
applications to representation theory.  Also, Bartels, Douglas, and Henriques
have recently announced a construction of Chern-Simons theory for 1-connected
compact Lie groups~\cite{D} which uses conformal nets and the cobordism
hypothesis.   

Raoul Bott was an inspiration to each of us, both personally and through his
mathematics.  We offer this paper as a tribute. 

  \section{The cohomology group $H^2(BG; {\mathbb Z})$ and one-dimensional
  theories}\label{sec:1}

To begin, we recall the definition of a topological quantum field theory.
Let $(\bordSO_n,\amalg)$ be the bordism category whose morphisms are oriented
$n$-manifolds; it carries a symmetric monoidal structure given by disjoint
union.  Let $(\VectC,\otimes )$ denote the symmetric monoidal category of
finite dimensional complex vector spaces under tensor product.

        \begin{definition}[]\label{thm:2}
  An \emph{$n$-dimensional topological quantum field theory}~$F$ on oriented
manifolds is a symmetric monoidal functor
  \begin{equation*}
     F\:(\bordSO_n,\amalg )\longrightarrow (\VectC,\otimes ). 
  \end{equation*}
        \end{definition}

\noindent
 Thus a field theory assigns to each closed oriented $(n-1)$-manifold~$Y$ a
complex vector space~$F(Y)$ and to each closed oriented $n$-manifold~$X$ a
number~$F(X)\in \CC$.  If $X\:Y_0\to Y_1$ is a compact manifold with
boundary\footnote{Here `$-Y_0$'~denotes the oppositely oriented
manifold.}~$-Y_0\amalg Y_1$, then $F(X)\:F(Y_0)\to F(Y_1)$ is a linear map.
The functor~$F$ maps the gluing operation of bordisms to the composition of
linear maps.  We refer the reader to~\cite{A}, \cite{S}, \cite{L} for more
details and exposition about this definition.

There are many possible variations on Definition~\ref{thm:2}.  For the domain
we may take a bordism category of manifolds with other topological
structures: oriented manifolds, spin manifolds, framed manifolds, etc.  There
are also other choices for the codomain.  This definition describes
``two-tiered theories''; we will see more tiers in~\S\ref{sec:2}.  If, say,
$n=3$, then we designate this a `2-3~theory' to emphasize the two tiers.
 
Let $n=1$.  Then a theory~$F$ assigns a vector space~$F(\pt_+)$ to the
positively oriented point.  We use the oriented interval with all
possibilities for decomposing the boundary into incoming and outgoing
components to deduce that $F(\pt_-)$~is the dual space of~$F(\pt_+)$ and the
intervals give the duality pairing.  Finally, $F(\cir)=\dim F(\pt)$, and the
entire theory is determined up to isomorphism by this nonnegative integer.


We now describe 1-dimensional pure gauge theory with compact gauge group~$G$
and action given by a class in~$H^2(BG;{\mathbb Z})$.  Let $\TT\subset \CC$ be the circle
group. 
        \begin{proposition}[]\label{thm:1}
 There is an canonical isomorphism $\Hom(G,\TT)\xrightarrow{\cong }H^2(BG; {\mathbb Z})$.
        \end{proposition}

\noindent
 In purely topological terms the classifying space~$BG$ carries a universal
principal $G$-bundle, and the isomorphism assigns to an abelian
character~$G\to\TT$ the Chern class of the associated principal $\TT$-bundle.
A more rigid viewpoint: abelian characters are in 1:1 correspondence with
isomorphism classes of principal $\TT$-bundle on the groupoid~$*\gpd G$.  A
proof that the latter are classified by~$H^2(BG; {\mathbb Z})$ may be found in several
texts, e.g.~\cite[Proposition~6.3]{AS}, \cite[Appendix]{Bry}.
 
Let $\lambda \:G\to\TT$ be an abelian character.  Then the \emph{classical}
1-dimensional gauge theory associated to~$\lambda $ assigns to each
$G$-bundle with connection on the circle the complex number~$\lambda (g)$,
where $g$~is the holonomy.  This is what physicists term the `exponentiated
action' of the theory.  Notice that holonomy depends on an orientation on the
circle, so this is a theory of \emph{oriented} 0-~and 1-manifolds.  To the
unique connection on a point is attached the trivial complex line~$\CC$; the
automorphism group~$G$ of the unique connection acts on~$\CC$ via~$\lambda $.
More precisely, this is correct if the point is positively oriented; for the
negatively oriented point the action is via~$\lambda \inv $.  Now the
standard quantization procedure constructs a theory~$F$ with $F(\pt_+)$~the
elements in~$\CC$ invariant under the action of the automorphisms.  (This is
called the `Gauss law' in physics.)  The result is the zero vector space
if~$\lambda \not= 1$ is nontrivial and is~$\CC$ if~$\lambda =1$ is trivial.
Then $F(\cir)=0$ in the first case and $F(\cir)=1$ in the second, as
$F(\cir)=\dim F(\pt_+)$.  These values may be understood as the result of the
path integral over the groupoid~$G\gpd G$ of connections on~$\cir$ with
respect to Haar measure:
  \begin{equation*}
     \frac{1}{\vol G}\int_{G} \lambda (g)\,dg =\begin{cases} 0,&\lambda
     \not= 1;\\1,&\lambda =1.\end{cases} 
  \end{equation*} 

        \begin{remark}[]\label{thm:4}
There are cohomology theories~$h$ with the property that $h^2(BG)$~is
isomorphic to the set of $\zt$-graded abelian characters~$(\lambda ,\epsilon
)$, i.e., $\lambda \:G\to \TT$ is an abelian character and $\epsilon =0,1$.
For example, we can take cohomology theory~$h$ determined by a spectrum with
two nontrivial homotopy groups: $\ZZ$~in degree~0 and~$\zt$ in degree~$-2$
connected by the nontrivial $k$-invariant.  A less efficient alternative is
the Anderson dual of the sphere~\cite[Definition~B.2]{HS}.  In either case
spin structures are required to define pushforward, so the corresponding
field theory is defined on spin manifolds.  It is then natural to let the
values of the theory lie in $\zt$-graded complex vector spaces.  If~$\epsilon
=0$, the theory factors through oriented manifolds and reduces to the
previous.  If~$\epsilon =1$, then the classical theory assigns the odd
line~$\CC$ to the unique connection on a point.  For~$\lambda \not= 1$ we
obtain the trivial theory as above, but now if~$\lambda =1$, then
$F(\cir_{\textnormal{n}})=-1$ and $F(\cir_{\textnormal{b}})=1$, where the
subscript indicates whether we consider the circle as endowed with the
nonbounding (n) or bounding (b) spin structure.

There are similar generalizations of the higher dimensional theories
discussed below.
        \end{remark}

        \begin{remark}[]\label{thm:5}
 All of the \emph{classical} theories alluded to in this paper are purely
topological and are most systematically defined by refining the class
in~$h^{n+1}(BG)$ to an object in the corresponding \emph{differential} theory
and then transgressing; see~\cite{F1} for the case~$n=3$.  The \emph{quantum}
theories, which are our focus, do not depend on the differential refinement.
        \end{remark}

  \section{The cohomology group $H^3(BG)$ and two-dimensional
theories}\label{sec:2} 

 \subsection{$H^3(BG; {\mathbb Z})$ for arbitrary~$G$}\label{subsec:2.1}

Analogous to Proposition~\ref{thm:1} we have the following. 

        \begin{proposition}[]\label{thm:6}
Let $G$~be a compact Lie group.  The cohomology group~$H^3(BG; {\mathbb Z})$
can be identified with the set of isomorphism classes of central extensions
$\TT\to\tG\to G$.
        \end{proposition}

\noindent
 In purely topological terms the class in~$H^3(BG; {\mathbb Z})$ attached to
a central extension $\TT\to\tG\to G$ is the obstruction to lifting the
universal $G$-bundle to a $\tG$-bundle.\footnote{One familiar case is
$\TT\to\Spin^c_n\to SO_n$ in which case the obstruction is the third
Stiefel-Whitney class~$W_3$.}  A central extension is equivalently a smooth
hermitian line bundle $K=\tK\to G$ together with isomorphisms
  \begin{equation}\label{eq:28}
     \theta _{x,y}\:K_x\otimes K_y\to K_{xy},\qquad x,y\in G,
  \end{equation}
for each pair in~$G$ which satisfy an associativity constraint  
  \begin{equation}\label{eq:29}
     \theta _{xy,z}\,(\theta _{x,y} \otimes \id_{K_z})= \theta
     _{x,yz}\,(\id_{K_x}\otimes \,\theta _{y,z} ),\qquad 
     x,y,z\in G,
  \end{equation}
for each triple in~$G$.  Recall the simplicial model
  \begin{equation}\label{eq:3}
     \xymatrix@1{\textnormal{$*$} \;\;&\;\; G\ar@<.5ex>[l] \ar@<-.5ex>[l]
     \;\;&\;\; G^2\ar@<1ex>[l] \ar[l] \ar@<-1ex>[l] \;\;&\;\; 
      G^3 \ar@<1.5ex>[l] \ar@<.5ex>[l] \ar@<-.5ex>[l]
     \ar@<-1.5ex>[l] \;\;&\;\; 
      G^4 \ar@<2ex>[l] \ar@<1ex>[l] \ar[l] \ar@<-1ex>[l]
     \ar@<-2ex>[l] \;\;\cdots} 
  \end{equation}
for~$BG$, which is quite useful for computing cohomology~\cite{B}.  Write
$S_p=G^p$ for the $p^{\textnormal{th}}$~space in~\eqref{eq:3}.  In these
terms a central extension is given by a line bundle $\tK\to S_1$, a
trivialization~\eqref{eq:28} of the alternating tensor product of its
pullback to~$S_2$, the latter constrained so that the alternating product of
its pullbacks to~$S_3$ is trivial~\eqref{eq:29}.  A proof of
Proposition~\ref{thm:6} may be found in~\cite[Proposition 6.3]{AS},
\cite[Part 1,\S2.2.1]{FHT}, \cite[Appendix]{Bry}.
 
Recall that the complex group algebra~$\CG$ of~$G$ is the abelian group of
complex functions on~$G$ with multiplication given by convolution with
respect to Haar measure, that is, pushforward under multiplication
$m\:G\times G\to G$.  If $\dim G>0$ then we take the continuous functions
on~$G$, which form a topological ring using the sup norm.  If $\dim G=0$
($G$~is finite) then the center of~$\CG$ is the commutative algebra with
natural basis the $\delta $-functions on the conjugacy classes.
Alternatively, it is the subalgebra of central functions, a description which
persists for all~$G$.

Associated to a central extension $\TT\to \tG\to G$ is a twisted complex
group algebra~$\tCG$, the algebra of sections of~$\tK\to G$ under
convolution.  Its center is again the commutative algebra of central
functions.

 \subsection{Gauge theory with finite gauge group}\label{subsec:1.3}

In an $n$-dimensional gauge theory with finite gauge group~$G$ the path
integral over a closed $n$-manifold~$X$ reduces to a finite sum, and this sum
defines the quantum invariant~$F(X)$.  The starting point of~\cite{F2} is an
expansion of this standard picture: the classical action can be extended to
manifolds of dimension~$<n$ and a higher categorical version of the path
integral determines the quantum invariants.  For any manifold~$M$ of
dimension~$\le n$, the groupoid of $G$-bundles may be identified with the
fundamental groupoid of the mapping space~$\Map(M,BG)$.  One can obtain a
larger class of theories by replacing~$BG$ with a topological space which has
a finite number of nonzero homotopy groups, each of which is a finite
group~\cite{Q}.  At least heuristically, the path integrals on manifolds of
dimension~$\le n$ define an extended field theory down to points which is
manifestly local and functorial.  We write these finite sums explicitly in
this section.  In~\S\ref{sec:5} we indicate one possible path which should
put these ideas on a rigorous footing.

For the 2-dimensional finite gauge theory based on the central extension
$\TT\to\tG\to G$ the classical invariant attached to a $G$-bundle $P\to X$
over a closed oriented 2-manifold is
  \begin{equation*}
     I(X,P) = e^{2\pi i \sigma _X(P)}, 
  \end{equation*}
where $\sigma \in H^2(BG;\QQ/\ZZ)\cong H^3(BG;\ZZ)$ is the characteristic
class corresponding to the central extension (see Proposition~\ref{thm:6})
and $\sigma _X(P)\in \QQ/\ZZ$ is the characteristic number of the bundle.
Physicists call~$I(X,P)$ the `exponentiated action' of the field~$P$ on~$X$.
The quantum invariant of~$X$ is then
  \begin{equation}\label{eq:31}
     F(X)= \sum\limits_{P}\frac{1}{\#\Aut P}\;I(X,P), 
  \end{equation}
where the sum is over a set of representative $G$-bundles on~$X$, one in each
equivalence class.  For the trivial central extension $I(X,P)=1$ for all~$P$
and \eqref{eq:31}~counts the number of representations of~$\pi _1X$ into the
finite group~$G$.  (There is an overall factor of~$1/\#G$.)  In that case we
can extend to a theory of unoriented manifolds.   
 
Now consider this 2-dimensional theory over the circle~$\cir$.  Recall from
the introduction that the groupoid of $G$-bundles over~$\cir$ is equivalent
to~$G\gpd G$, where $G$~acts on itself by conjugation.  The value of the
classical theory~$I$ is given by the central extension $\tG\to G$, viewed as
an equivariant principal $G$-bundle.  In other words, the value~$I(\cir,P)$
of this theory at a bundle with holonomy $x\in G$ is the circle
torsor~$\tG_x$ (together with the action of the centralizer of~$x$, the
automorphisms of~$P$, by conjugation in the central extension.)  In the
quantum theory we make a sum of the corresponding hermitian lines~$K_P$
analogous to~\eqref{eq:31} and so compute that $F(\cir)$~is the vector space
of central sections of $\tK\to G$:
  \begin{equation}\label{eq:32}
     F(\cir) = \bigoplus\limits_{P}\,\frac{1}{\#\Aut P}\,(K_P)^{\Aut P}  
  \end{equation}
In this expression the metric on the hermitian line~$K_x$ is scaled by the
prefactor.
 
The quantum ``path integrals''~\eqref{eq:31} and~\eqref{eq:32} may be
expressed in categorical language.  Let $\bordSOG$ denote the bordism
category whose objects are closed oriented 1-manifolds equipped with a
principal $G$-bundle and whose morphisms are compact oriented 2-manifolds
equipped with a principal $G$-bundle.  Then the classical topological
theory~$I$ is a 2-dimensional topological quantum field theory in the sense
of Definition~\ref{thm:2} with a special feature: its values are
\emph{invertible}.  Thus the numbers attached to closed oriented surfaces
with $G$-bundle are nonzero and the vector spaces attached to $G$-bundles
over the circle are lines (which are invertible in the collection of vector
spaces under tensor product).  Now there is an obvious forgetful functor $\pi
\:\bordSOG\to\bordSO_2$ which omits the $G$-bundle.  Then the quantum field
theory $F$, obtained by summing
 over $G$-bundles, can be viewed as a kind of pushforward of $I$ along $\pi$.
The relevant pushforward procedure has an analog in classical topology.
Let $\pi \:E\to S$ be a proper fiber bundle of topological manifolds whose
fibers carry a suitable $n$-dimensional orientation.  Then there is a map
$\pi _*\:h^{\bullet }(E)\to h^{\bullet -n}(S)$ variously termed the
`pushforward' or `direct image' or `wrong-way map' or `\emph{umkehr} map' or
`Gysin map' or `transfer'.  The analogy with our situation is tighter if we
think of this pushforward on the level of cochains, or geometric
representatives, rather than cohomology classes.

Formulas~\eqref{eq:31} and~\eqref{eq:32}, together with a similar formula for
compact surfaces with boundary, can be used to define the functor~$F$.  This
constructs a 1-2~theory.  We would like, however, to continue down to points,
i.e., to a 0-1-2 theory.  To express the higher gluing laws encoded in an
$n$-dimensional TQFT which extends down to points we use the language of
higher category theory; see~\cite{L} for an introduction and for much more
exposition about the $n$-categories we now introduce.  Let $\BordSO_n$ denote
the $n$-category\footnote{$\BordSO_n$~should be regarded as an $(\infty
,n)$-category.  This means roughly that we consider families of manifolds
parametrized by a topological space rather than simply single manifolds.}
whose objects are finite unions of oriented points, 1-morphisms are oriented
bordisms of objects, 2-morphisms are oriented bordisms of 1-morphisms, and so
forth.\footnote{An $(\infty ,n)$-category~has $r$-morphisms for all~$r$; they
are invertible for~$r>n$.  The $r$-morphisms in~ $\BordSO_n$ for $r > n$ are
given by $(r-1)$-parameter families of diffeomorphisms of manifolds which
preserve all boundaries (and corners).} Note $\BordSO_n$~carries a symmetric
monoidal structure given by disjoint union.  Let $\C$~be any symmetric
monoidal $n$-category.

        \begin{definition}[]\label{thm:23}
 An \emph{extended $n$-dimensional TQFT with values in~$\C$\/} is a symmetric
monoidal functor $F\:\BordSO_n\to\C$. 
        \end{definition}
 
The 0-1-2 finite gauge theory can be constructed by hand.

 \subsection{0-1-2 theories for general~$G$}\label{subsec:2.2}

Let $F$ be a fully extended topological field theory of dimension $n$.
We can regard $F$ as a prescription for assigning invariants to
manifolds with corners of all dimensions $\leq n$, together with set of
rules for how these invariants behave when we glue manifolds together.
Since every smooth manifold $M$ can be assembled by gluing together very simple pieces
(for example, by choosing a triangulation of $M$), we might imagine that the value of
$F$ on arbitrary manifolds is determined by its values on a very small class
of manifolds. In order to formulate this idea more precisely, we need to introduce a bit of terminology.

\begin{definition}
Let $M$ be a manifold of dimension $m \leq n$. An {\it $n$-framing} of $M$ is a trivialization
of the vector bundle $T_{M} \oplus \underline{ \RR^{n-m} }$, where
$\underline{ \RR^{n-m} }$ denotes the trivial bundle (on $M$) of rank $n-m$.
We let $\Bord_{n}^{\fr}$ denote the bordism $n$-category whose $k$-morphisms
are given by $n$-framed $k$-manifolds for $k \leq n$.
\end{definition}

\begin{theorem}[Baez-Dolan Cobordism Hypothesis]\label{bd}
Let $\C$ be a symmetric monoidal $n$-category. Then the construction
$$ F \mapsto F(\ast)$$
induces an injection from the collection of isomorphism classes of
symmetric monoidal functors $F: \Bord_{n}^{\fr} \rightarrow \C$
to the collection of isomorphism classes of objects of $\C$.
\end{theorem}

A version of Theorem \ref{bd} was originally conjectured by Baez and Dolan in
\cite{BD}. We refer the reader to \cite{L} for a more extensive discussion
and a sketch of the proof.

Let $\C$ be a symmetric monoidal $n$-category. We will say that an object
$C \in \C$ is {\it fully dualizable} if there exists an extended TQFT
$F: \Bord_{n}^{\fr} \rightarrow \C$ and an isomorphism $C \simeq F(\ast)$.
Theorem \ref{bd} asserts that if $C$ is a fully dualizable object of $\C$, then
the field theory $F$ is uniquely determined by~ $C$.

\begin{remark}
It is possible to state a more precise version of Theorem \ref{bd} by
describing the class of fully dualizable objects $C \in \C$ without mentioning
the $n$-category $\Bord_{n}^{\fr}$. If $n=1$, this condition is easy to state:
the object $C$ should admit a dual $C^{\vee}$, so that there exist evaluation and coevaluation
maps
$$ C \otimes C^{\vee} \stackrel{e}{\rightarrow} {\bf 1} \quad \quad { \bf 1} \stackrel{c}{\rightarrow} C^{\vee} \otimes C$$
(here ${\bf 1}$ denotes the unit object of $\C$) which are compatible in the sense that
the compositions
$$ C \rightarrow C \otimes C^{\vee} \otimes C \rightarrow C$$
$$ C^{\vee} \rightarrow C^{\vee} \otimes C \otimes C^{\vee} \rightarrow C^{\vee}$$
both coincide with the identity. For $n > 1$, we need to assume that analogous finiteness
assumptions are satisfied not only by the object $C$, but by the $1$-morphisms $e$ and $c$.
We refer the reader to \cite{L} for a more complete discussion.
\end{remark}

Theorem \ref{bd} has a curious consequence: since the orthogonal group
$\OO(n)$ acts on the $n$-category $\Bord_{n}^{\fr}$ (by change of framing),
we get an induced action of $\OO(n)$ on the collection of fully dualizable
objects of any symmetric monoidal $n$-category $\C$ (more precisely, the
orthogonal group $\OO(n)$ can be made to act on the classifying space for the
underlying $n$-groupoid of the fully dualizable objects in $\C$). When $n=1$,
this action is simply given by the involution that takes a dualizable object
$C \in \C$ to its dual $C^{\vee}$.

We can use the action of the orthogonal group to formulate an analog of Theorem
\ref{bd} for more general types of manifolds:

\begin{theorem}\label{bdo}
Let $\C$ be a symmetric monoidal $n$-category. The construction $F \mapsto F(\ast)$ establishes
a bijection between the set of isomorphism classes of symmetric monoidal functors
$\Bord_{n}^{\SO} \rightarrow \C$ with the set of isomorphism classes of
(homotopy) fixed points for the action of the group $\SO(n)$ on the fully dualizable objects of $\C$.
\end{theorem}

\begin{remark}
Theorem \ref{bdo} has an obvious analog for other types of manifolds: unoriented manifolds, spin manifolds, and so forth.
\end{remark}

	\begin{example}\label{alg}
 Let $n=2$, and let $\C$ denote the $2$-category $\Alg$ of (complex)
algebras, bimodules, and intertwiners. Every object of $\C$ admits a dual:
the dual of an algebra $A$ is the opposite algebra $A^{op}$, where both the
evaluation and coevaluation maps are given by $A$ (regarded as a $({\mathbb
C}, A \otimes A^{op})$-bimodule.  An algebra $A$ is fully dualizable if and
only if $A$ is dualizable both as a ${\mathbb C}$-module and as an $A \otimes
A^{op}$-module. The first condition amounts to the requirement that $A$ be
finite dimensional over ${\mathbb C}$, while the second condition requires
that the algebra $A$ be semisimple.

The circle group $\SO(2)$ acts on the classifying space of the $2$-groupoid
of fully dualizable objects of $\Alg$. In more concrete terms, this means
that every fully dualizable object $A \in \Alg$ determines a functor from the
fundamental $2$-groupoid of $\SO(2)$ into $\Alg$ which carries the identity
element of $\SO(2)$ to $A$. Applying this to a generator of the fundamental
group $\pi_1 \SO(2)$, we get an automorphism of $A$ in $\Alg$: this
automorphism is given by the vector space dual $A^{\vee}$, regarded as an
$(A,A)$-bimodule. To realize $A$ as a fixed point for the action of $\SO(2)$,
we need to choose an identification of $A$ with $A^{\vee}$ as
$(A,A)$-bimodules. In other words, we need to choose a nondegenerate bilinear
form $b: A \otimes A \rightarrow {\mathbb C}$ which satisfies the relations
$$ b(xa,a') = b(a,a'x) \quad \quad b(ax, a') = b(a,xa').$$
If we set $\tr(a) = b(1,a)$, then the first condition shows that $b(a,a') =
\tr(a'a)$, while the second condition shows that $b(a,a') = \tr(aa')$.  It
follows that $\tr$ is a {\it trace} on the algebra $A$: that is, it vanishes
on all commutators $[a,a'] = aa' - a'a$. Conversely, given any linear map
$\tr: A \rightarrow {\mathbb C}$ which vanishes on all commutators, the
formula $b(a,a') = \tr(aa')$ defines a bilinear form $b$ giving a map of
$(A,A)$-bimodules $A \rightarrow A^{\vee}$. We say that $\tr$ is {\it
nondegenerate} if this map is an isomorphism. A pair $(A, \tr)$ where $A$ is
a finite dimensional algebra over ${\mathbb C}$ and $\tr$ is a nondegenerate
trace on $A$ is called a {\it Frobenius algebra}.

We can summarize the above discussion as follows: giving a fully dualizable
object of $\Alg$ which is fixed under the action of the group $\SO(2)$ is
equivalent to giving a semisimple Frobenius algebra $(A,\tr)$. Theorem
\ref{bdo} implies that every such pair $(A,\tr)$ determines an extended TQFT
$F: \Bord_{2}^{\SO} \rightarrow \Alg$ such that $F(\ast) \simeq A$.  The
value of $F$ on the circle $S^1$ can be identified with the $( {\mathbb C},
{\mathbb C})$-vector space given by the tensor product $A \otimes_{ A \otimes
A^{op} } A$: in other words, the quotient of $A$ by the subspace $[A,A]$
generated by all commutators. Using the self-duality of $F(S^1)$ provided by
the field theory $F$ and the self-duality of $A$ provided by the bilinear
form $b$, we can also identify $F(S^1)$ with the orthogonal $[A,A]^{\perp}$
of $[A,A]$ with respect to $b$: that is, with the center of the algebra $A$.

In particular, if $G$ is a finite group and $\eta \in H^3(G; \ZZ)$ is a
cohomology class, then the twisted group algebra $\tCG$ admits a canonical
trace (obtained by taking the coefficient of the unit in $G$ and dividing by
the order of $G$).  Applying Theorem \ref{bdo} in this case, we obtain
another construction of the topological field theory described
in~\ref{subsec:1.3}.
	\end{example}

  \section{Finite path integrals}\label{sec:5}

 We sketch an idea to construct extended field theories via finite sums.

        \begin{remark}[]\label{thm:20}
 Integration over a finite field, such as a gauge field with finite gauge
group, sometimes occurs in a quantum field theory with other fields.  These
cases also fit into this framework.  Examples include orbifolds and
orientifolds in string theory.
        \end{remark}

As
in~\cite[\S3.2]{L} let $\famn$ denote the $n$-category whose objects are
finite groupoids~$X$.  (A groupoid~$X$ is finite if there is a finite number
of inequivalent objects and each object has a finite automorphism group.)  A
1-morphism $C\:X\to Y$ between finite groupoids is a correspondence
  \begin{equation}\label{eq:36}
     \xymatrix{&C\ar[dl]_{p_1}\ar[dr]^{p_2}\\X&&Y} 
  \end{equation}
of finite groupoids.  A 2-morphism in $\Fam_n$ is a correspondence of
1-morphisms, and so forth until the level $n$; we regard two $n$-morphisms in
$\Fam_n$ as identical if they are equivalent.  Composition is homotopy fiber
product. 
Cartesian product of groupoids endows~ $\famn$ with a symmetric monoidal structure.  There is a
symmetric monoidal functor
  \begin{equation}\label{eq:37}
     \BunG\:\BordSO_n\longrightarrow \famn 
  \end{equation}
which attaches to each manifold~$M$ the finite groupoid of $G$-bundles
on~$M$.  This is the space of classical fields on~$M$.  (The
functor~\eqref{eq:37} replaces the category~$\bordSO_n(G)$ which appears in
our previous formulation in the paragraph following~\eqref{eq:32}.)  Let
$\C$~be a symmetric monoidal $n$-category and $\famn(\C)$~the symmetric
monoidal $n$-category of correspondences equipped with local systems valued
in~$\C$: for example, an object of~$\famn(\C)$ is a finite groupoid~$X$ and a
functor $X\to\C$, and morphisms are also equipped with functors to~$\C$ (as
before, $\C$ denotes the codomain of our field theories).  Then the classical
theory is encoded by a functor~$I$ which fits into a commutative diagram
  \begin{equation*}
     \xymatrix{\BordSO_n\ar[rr]^I \ar[dr]_{\BunG}&&\famn(\C) \ar[dl]\\ &\famn}
  \end{equation*}
where the right arrow is the obvious forgetful functor.   

        \begin{remark}[]\label{thm:25}
 For a \emph{classical} theory the values of~$I$ lie in the \emph{invertible}
objects and morphisms of~$\C$.  The finite sum construction does not depend
on the invertibility, so will apply in the situations envisioned in
Remark~\ref{thm:20}.  
        \end{remark}

        \begin{example}[]\label{thm:24}
 We recast the discussion in~\ref{subsec:2.1} in these terms. As a
preliminary suppose $\B,\C$~are 2-categories.  We describe the data of a
functor~$I\:\B\to\C$.  Each 2-category consists of objects (ob), 1-morphisms
(1-mor), and 2-morphisms (2-mor) and the data which defines $I$~includes a
map between these corresponding collections in~$\B$ and~$\C$.  For
\emph{strict} 2-categories these maps are required to respect the composition
strictly.  But for more general 2-categories there are two additional pieces
of data.  First, there is a map $u\:\ob(\B)\to \2-mor(\C)$ which for each
object $X\in \ob(\B)$ gives a 2-morphism $u(X)\:\id_{I(X)}\to I(\id_{X})$.
Second, there is a map $a\:\1-mor(\B)\times
_{\ob(\B)}\1-mor(\B)\to\2-mor(\C)$ which expresses the failure of $I$~on
1-morphisms to be a strict homomorphism, namely $a(x,y)\:I(y)\circ I(x)\to
I(y\circ x)$ for every pair~$\cdot \xrightarrow{x}\cdot \xrightarrow{y}\cdot
$ of composable 1-morphisms in~$\B$.  These data are required to obey a
variety of axioms.  For example, if $\cdot \xrightarrow{x}\cdot
\xrightarrow{y}\cdot \xrightarrow{z}\cdot $~is a triple of composable
morphisms, then
  \begin{equation}\label{eq:44}
     a(y\circ x,z)\circ \{a(x,y)*\id\mstrut _{I(z)}\}= a(x,z\circ y)\circ
     \{\id\mstrut _{I(x)}*\,a(y,z)\}.  
  \end{equation}

As in Example~\ref{alg}, let $\C=\Alg$~be the 2-category whose objects are
complex algebras~$A$.  A morphism $A\to A'$ is an $\bimod {A'}A$-bimodule and
a 2-morphism is a homomorphism of bimodules.  The symmetric monoidal
structure is given by tensor product over~$\CC$. Note that the unit object
in~$\C$ is the algebra~$\CC$ and a $\bimod{\CC}{\CC}$-bimodule is simply a
vector space.  Thus the category of 1-morphisms $\CC\to\CC$ is the category
$\Vect_{\CC}$ of complex vector spaces.  Now for the finite 2-dimensional
gauge theory based on the central extension $\TT\to \tG\to G$, the
functor~$\BunG$ assigns to the point~$\pt_+$ the groupoid~$*\gpd G$.  The
lift~$I$ includes the functor
  \begin{equation}\label{eq:46}
      \chi \:*\gpd G\to\Alg 
  \end{equation}
which assigns the (invertible) algebra~$\CC$ to the unique object~$*$, the
complex line~$K_x$ to the 1-morphism~$x\in G$, and the identity map to the
identity 2-morphisms in~$*\gpd G$.  (There are only identity 2-morphisms
in~$*\gpd G$.)  The map~$u$ in the previous paragraph is then the identity.
The map~$a$ in the previous paragraph assigns to each pair~$x,y$ of group
elements the isomorphism~\eqref{eq:28}, viewed as a 2-morphism in~$\C$.  The
associativity constraint~\eqref{eq:44} is~\eqref{eq:29}.
        \end{example}

The quantization via a finite sum will be implemented by a symmetric monoidal
functor of $n$-categories (see~\S\ref{cs8})
  \begin{equation}\label{eq:39}
     \Sum_n\:\famn(\C)\longrightarrow \C 
  \end{equation}
Given~\eqref{eq:39} we simply define the quantum theory as the composition
  \begin{equation}\label{eq:3.9}
     F\:\BordSO_n\xrightarrow{\quad I\quad
     }\famn(\C)\xrightarrow{\;\;\Sum_n\;\;}\C  
  \end{equation}
The functor $\Sum_n$ depicted in~\eqref{eq:39} is given by a purely categorical
procedure: if $X$ is a finite groupoid---an object in~$\famn$---and $\chi: X
\rightarrow \C$ is a $\C$-valued local system on $X$, then $\Sum_n(X, \chi)$ is
given by the colimit $\varprojlim_{x \in X} \chi(x)$.  To guarantee that this
formula describes a well-defined functor from $\famn(\C)$ to $\C$, we need to
make certain assumptions on $\C$: namely, that it is additive in a strong
sense which guarantees that the colimit $\varinjlim_{x \in X} \chi(x)$ exists
and coincides with the limit $\varprojlim_{x \in X} \chi(x)$.

        \begin{remark}[]\label{thm:31}
 This discussion goes through if we replace~$\famn$ by the $n$-category whose
objects are finite $n$-groupoids (see~\S\ref{cs8}).  We will use the
generalization to~$n=2$ in~\ref{subsec:6.1}.
        \end{remark}

        \begin{example}[]\label{thm:26}
 To illustrate the idea consider the 1-dimensional gauge theory
of~\S\ref{sec:1} for a finite gauge group~$G$.  Recall it is specified by an
abelian character~$\lambda \:G\to\TT$.  In this case let $\C$~be the
symmetric monoidal category~$\Vect_{\CC}$ with tensor product.  Then the
``classical'' functor $I\:\BordSO_1\to\Fam_1(\Vect_{\CC})$ sends the
point~$\pt_+$ to the functor $*\gpd G\to\Vect_{\CC}$ which sends~$*$ to~$\CC$
and is the homomorphism~$\lambda $ on morphisms.  The path
integral~\eqref{eq:39} is defined as follows.  If $X\in \Fam_1(\Vect_\CC)$ is
a finite groupoid equipped with a functor~$\chi \:X\to\Vect_{\CC}$ then
$\Sum_1(X,\chi )\in \Vect_{\CC}$ is the limit
  \begin{equation}\label{eq:41}
     \Sum_1(X,\chi ) = \varprojlim\limits_{x\in X} \chi (x) 
  \end{equation}
If the finite groupoid $X=\xymatrix@1{X_0 & \ar@<.5ex>[l] \ar@<-.5ex>[l]
X_1}$~is finitely presented ($X_0$~and $X_1$ are finite sets), then
$\chi $~determines an equivariant vector bundle over~$X_0$ and the
limit~\eqref{eq:41} is the vector space of invariant sections.  On the other
hand, a morphism in~$\Fam_1(\Vect_{\CC})$ is given by a
correspondence~\eqref{eq:36} of finite groupoids; functors $\chi
\:X\to\Vect_{\CC}$ and $\delta \:Y\to\Vect_{\CC}$; and for each~$c\in C$ a
linear map $\varphi (c)\:\chi \bigl(p_1(c) \bigr)\to \delta \bigl(p_2(c)
\bigr)$.  We define a map
  \begin{equation*}
     \Sum_1(C,\varphi )\:\Sum_1(X,\chi )\longrightarrow \Sum_1(Y,\delta ) 
  \end{equation*}
Assume $X,Y,C$~are finitely presented.  Given~$x\in X$ let $C_x$~be the
sub-groupoid of~$C$ consisting of $c\in C$ such that $p_1(c)=x$.  Then  
  \begin{equation}\label{eq:43}
     \Sum_1(C,\varphi ) = \sum\limits_{[c]\in \pi _0C_x}\frac{\varphi
     (c)}{\#\Aut(c)} 
  \end{equation}
where the sum is over equivalence classes of objects in~$C_x$.
(Compare~\eqref{eq:31}.)  The reader can check that \eqref{eq:41}
and~\eqref{eq:43} reproduce the results in the paragraph preceding
Remark~\ref{thm:4} in the 1-dimensional finite gauge theory.
        \end{example}


        \begin{example}[]\label{thm:27}
 We continue Example~\ref{thm:24} and content ourselves with the computation
of~$F(\pt_+)$ as a limit in the 2-category~$\C=\Alg$:  
  \begin{equation*}
     F(pt_+) = \varprojlim\limits_{*\gpd G} \chi , 
  \end{equation*}
where $\chi $~is the functor~\eqref{eq:46}.  Unraveling the definitions, this
limit is given by an algebra~$A$ with the following universal property: for any
algebra $B$, category of $(A,B)$-bimodules is equivalent to the category of
right $B$-modules $M$ equipped with a compatible family of $B$-module isomorphisms
$\{ K_{x} \otimes M \simeq M \}_{x \in G}$.  This limit can be represented by the twisted group algebra
  \begin{equation*}
     A=\tCG = \bigoplus \limits_{x\in G}K_x.
  \end{equation*}
 
We leave as an exercise to the reader the computation of~$F(\cir)$ as a
(1-categorical) limit over the groupoid~$G\gpd G$ of $G$-bundles on~$\cir$.
The argument is similar to Example~\ref{thm:26}.
        \end{example}

  \section{Three-dimensional theories with finite gauge group}\label{sec:3} 

 \subsection{$H^4(BG; {\mathbb Z})$ for finite~$G$}\label{subsec:1.4}
 
Let $G$~be a finite group.   

        \begin{definition}[]\label{thm:11}
  A \emph{2-cocycle} on~$G$ with values in hermitian lines is a
pair~$(\tK,\ttheta )$ consisting of a hermitian line bundle $\tK\to G\times
G$, for each triple $x,y,z\in G$ an isometry
  \begin{equation}\label{eq:4}
     \theta \mstrut _{x,y,z}\:\mstrut K_{y,z}\otimes K_{xy,z}\inv \otimes
     K\mstrut _{x,yz}\otimes K_{x,y}\inv \longrightarrow \CC 
  \end{equation}
and a cocycle condition 
  \begin{equation}\label{eq:11}
     \theta \mstrut _{y,z,w}\,\theta _{xy,z,w}\inv \,\theta \mstrut
     _{x,yz,w}\,\theta _{x,y,zw}\inv \,\theta \mstrut _{x,y,z}= 1,\qquad
     x,y,z,w\in G. 
  \end{equation}
for each quadruple of elements of~$G$.   
        \end{definition}

        \begin{proposition}[]\label{thm:7}
 For $G$~finite the cohomology group~$H^4(BG; \ZZ)$ is the set of isomorphism
classes of 2-cocycles~$(\tK,\ttheta)$ on~$G$ with values in hermitian lines.
        \end{proposition}

        \begin{proof}
 Since $H^4(BG;\RR)=0$, the Bockstein homomorphism $H^3(BG;\TT)\to
H^4(BG;\ZZ)$ from the exponential sequence of coefficients is an
isomorphism.\footnote{Here the circle~$\TT$ has either the discrete topology
or the continuous topology.}  Given~$(\tK,\ttheta)$ choose $k_{x,y}\in
K_{x,y}$ of unit norm.  Then
  \begin{equation*}
     \omega \mstrut _{x,y,z}=\theta \mstrut _{x,y,z}(k\mstrut
     _{y,z}\,k_{xy,z}\inv \,k\mstrut _{x,yz}\,k_{x,y}\inv )
  \end{equation*}
is a 3-cocycle with values in~$\TT$.  A routine check shows that the
resulting element of~$H^3(BG;\TT)$ is independent of~$\{k_{x,y}\}$ and of the
representative ~$(\tK,\ttheta)$ in an equivalence class.  Also, the resulting
map from equivalence classes of 2-cocycles to~$H^3(BG;\TT)$ is an injective
homomorphism.  Surjectivity is also immediate: if $\omega _{x,y,z}$~is a
cocycle, set $K_{x,y}=\CC$ and $\theta _{x,y,z}=\omega _{x,y,z}$.
        \end{proof}

 Suppose $(\tK,\ttheta)$~is given.  Define the line bundle $L\to G\times G$
by
  \begin{equation}\label{eq:6}
     L\mstrut _{x,y} = K^*_{yxy\inv ,y}\otimes K\mstrut _{y,x}\;. 
  \end{equation}
The cocycle isomorphism~\eqref{eq:4} leads to an isomorphism 
  \begin{equation*}
     L_{yxy\inv,y' }\otimes L_{x,y}\longrightarrow L_{x,y'y} 
  \end{equation*}
which is summarized by the statement that $L$~is a line bundle over the
groupoid~$G\gpd G$ formed by the $G$~action on itself by conjugation.
In~\ref{subsec:1.5} the line bundle~$K$ enters into the quantization of a
point and the line bundle~$L$ enters into the quantization of the circle.

 \subsection{Three-dimensional finite gauge theory}\label{subsec:1.5}

Let $G$~be finite and fix a 2-cocycle~$(\tK,\ttheta)$ with values in
hermitian lines.  Let $\tVect[G]$~be the category whose objects are complex
vector bundles over~$G$ and morphisms are linear vector bundle maps.  Define
a monoidal structure on~$\tVect[G]$ by twisted convolution: if $W,W'\to G$
are vector bundles set
  \begin{equation*}
     (W* W')_y = \bigoplus\limits_{xx'=y} K\mstrut _{x,x'}\otimes W\mstrut
     _x\otimes W'_{x'}. 
  \end{equation*}
Then $\tVect[G]$~is a linear tensor category.   

        \begin{remark}[]\label{thm:32}
 One way to regard this
category is as the ``crossed product'' $T\ltimes \vect$, with the tensor
category $\vect$ replacing the customary algebra with $T$-action. (Actions of
$T$ on $\vect$, as a tensor category, assign to each $t\in T$ an invertible
$(\vect,\vect)$-bimodule, which must of course be isomorphic to $\vect$.  The
associator for the action is a $2$-cocycle on $T$ with values in $\pic$, the
group of units in $\vect$; and associators for equivalent actions differ by a
co-boundary. Our action is classified by $\tau$, as per the discussion
in~\ref{subsec:1.4}.)
        \end{remark}

There is a TQFT~$\tF$ such that $\tF(\pt)=\tVect[G]$ and it can be
constructed in several ways.  First, and most directly, we can construct it
as a 0-1-2-3~theory directly using the finite sum path integral
in~\S\ref{sec:5}.  Secondly, following Reshetikhin and Turaev~\cite{T}, as a
1-2-3 theory it may be constructed by specifying the \emph{modular tensor
category} which is attached to a circle; it appears in
Proposition~\ref{thm:15} below. A third approach is to realize $\tVect[G]$ as
a fully dualizable object of a symmetric monoidal $3$-category $\C$. The Baez-Dolan cobordism hypothesis then implies that $\tVect[G]$ determines a 0-1-2-3~theory which is defined on
framed $3$-manifolds. To remove the dependence on a choice of framing, we should go further and exhibit $\tVect[G]$ as an {\it $SO(3)$-equivariant} fully dualizable object. For present purposes, we will be content to sketch a definition of the relevant $3$-category $\C$ and to give some hints at what the relevant finiteness conditions correspond to.

 \begin{definition}\label{canst} The $3$-category $\C$ can be described
informally as follows:
 \begin{itemize}
 \item[$(a)$] The objects of $\C$ are tensor categories over~$\CC$: in other
words, $\CC$-linear categories equipped with a tensor product operation which
is associative up to coherent isomorphism.
 \item[$(b)$] Given a pair of tensor categories $\A$ and $\A'$, a
$1$-morphism from $\A$ to $\A'$ in $\C$ is an {\it $\A$-$\A'$ bimodule
category}: that is, a $\CC$-linear category $\D$ equipped with a left action
$\A \times \D\to\D$ and a right action $\D \times \A'\to\D$ which commute
with one another, up to coherent isomorphism.
 \item[$(c)$] Given a pair of tensor categories $\A$ and $\A'$ and a pair of
bimodule categories $\D$ and $\D'$, a $2$-morphism from $\D$ to $\D'$ in $\C$
is a functor between bimodule categories: that is, a functor $F: \D
\rightarrow \D'$ which commutes with the actions of $\A$ and $\A'$ up to
coherent isomorphism.
 \item[$(d)$] Given a pair of tensor categories $\A$ and $\A'$, a pair of
bimodule categories $\D$ and $\D'$, and a pair of bimodule category functors
$F,F': \D \rightarrow \D'$, a $3$-morphism from $F$ to $F'$ in $\C$ is a
natural transformation $\alpha: F \rightarrow F'$ which is compatible with
the coherence isomorphisms of $(c)$.
 \end{itemize}
 \end{definition}

The category $\tVect[G]$ is a fully dualizable object of $\C$. Roughly
speaking, the verification of this takes place in three (successively more
difficult) steps. First, we verify that $\tVect[G]$~ is dualizable: in other
words, that it is a fully dualizable object of the underlying $1$-category of
$\C$ (and therefore gives rise to a $1$-dimensional field theory). This is
completely formal: every tensor category $\A$ is a dualizable object of $\C$,
the dual being the same category with the opposite tensor product. The next
step is to verify that $\tVect[G]$ is a fully dualizable object of the
underlying $2$-category of $\C$ (and therefore gives rise to a
$2$-dimensional field theory, which assigns vector spaces to surfaces). This
is a consequence of the fact that $\tVect[G]$ is a {\it rigid} tensor
category: that is, every object of $X \in \tVect[G]$ has a dual (given by
taking the dual vector bundle $X^{\vee}$ of $X$ and pulling back under the
inversion map $g \mapsto g^{-1}$ from $G$ to itself). Finally, to get a
$3$-dimensional field theory, we need to check that $\tVect[G]$ satisfies
some additional $3$-categorical finiteness conditions which we will not spell
out here. (However, we should remark that these $3$-categorical finiteness
conditions are in some sense the most concrete, and often amount to the
finite dimensionality of various vector spaces associated to $\tVect[G]$: for
example, the vector spaces of morphisms between objects of $\tVect[G]$.)  The
paper~\cite{BW} is presumably relevant to the full dualizability
of~$\tVect[G]$.

        \begin{definition}[]\label{thm:14}
  Let $\A$~be a monoidal category with product~$*$.  Its \emph{(Drinfeld)
center}~$Z(\A)$ is the category whose objects are pairs~$(X,\epsilon_X )$
consisting of an object~$X$ in~$\A$ and a natural isomorphism $\epsilon_X
(-)\:X* -\to -* X$.  The isomorphism~$\epsilon_X $ is compatible with the
monoidal structure in that for all objects~$Y,Z$ in~$\A$ we require
  \begin{equation*}
     \epsilon_X (Y* Z)= \bigl(\id_Y* \epsilon_X (Z) \bigr)\circ
     \bigl(\epsilon_X (Y)* \id_Z \bigr). 
  \end{equation*}
        \end{definition}
\indent 
 The center~$Z(\A)$ of any monoidal category~$\A$ is a \emph{braided}
monoidal category.  M\"uger~\cite{M} proves that if~$\A$ is a linear tensor
category over an algebraically closed field which satisfies certain
conditions, then $Z(\A)$~is a modular tensor category.  That applies to
part~(ii) of the following result.

 \begin{proposition}[]\label{thm:15}
 \begin{enumerate}[(i)]
 \item The value $\tF(\cir)$ of the field theory~$\tF$ on the circle is the
center of the monoidal category~$\tVect[G]$.
 \item The center of~$\tVect[G]$ consists of twisted equivariant vector
bundles $W\to G$, that is, vector bundles with a twisted lift
  \begin{equation}\label{eq:24}
     L_{x,y}\otimes W_x\longrightarrow W_{yxy\inv } 
  \end{equation}
of the $G$-action on~$G$ by conjugation, where $L\to G\times G$ is defined
in~\eqref{eq:6}. 
 \end{enumerate} 
        \end{proposition}

\noindent 
 This center is the well-known modular tensor category attached to a circle
in the twisted finite group Chern-Simons theory~\cite{F2}.
   
     \begin{proof}
We compute~$\tF(\cir)$ by decomposing the circle into two intervals~$I_L$
and~$I_R$.  We regard~$I_L$ as a morphism in the oriented bordism category
from the empty set to the disjoint union of two oppositely oriented points;
$I_R$~is a bordism in the other direction.  Then $\cir$~is the composition
$I_R\circ I_L$.  Let $A=\tVect[G]$, which is the object in~$\C$---a tensor
category---attached to a positively oriented point.  Then the tensor
category~$\Aop$ with the opposite monoidal structure is attached to the
negatively oriented point.  Now $\tF(I_L)$~is $A$ viewed as a left module
for~$A\otimes \Aop$ and $\tF(I_R)$~is $A$ viewed as a right module
for~$A\otimes \Aop$.  Thus
  \begin{equation*}
     \tF(\cir)\cong \tF(I_R\circ I_L)\cong  A \otimes \mstrut _{A\otimes
     \Aop}A 
  \end{equation*}
(This is by definition the Hochschild homology of~$A$.)  But now we use
additional structure on~$\tVect[G]$ which gives an isomorphism of~$A$ with
its linear dual $\dA=\Hom(A,\Vect)$.  Namely, there is a trace $\theta
\:A\to \Vect$ which maps a vector bundle over~$G$ to the fiber over the
identity element, and the corresponding bilinear form 
  \begin{equation*}
     \begin{split} A \otimes A &\longrightarrow \Vect \\ W\otimes W'
      &\longmapsto \theta (W*W')\end{split} 
  \end{equation*}
induces the desired identification.  Therefore, 
  \begin{equation*}
     \tF(\cir)\cong A\otimes \mstrut _{A\otimes \Aop}\dA \cong \Hom\mstrut
     _{A\otimes \Aop}(A,A) 
  \end{equation*}
which we may identify with the center of~$A$ (the Hochschild cohomology).

For~(ii) suppose $W\to G$ is in the center.  Let $W'\to G$ be the vector
bundle which is the trivial line~$\CC_y$ at some~$y\in G$ and zero elsewhere.
Then the braiding gives, for every~$x\in G$, an isomorphism
  \begin{equation*}
     K_{y,x}\otimes \CC_y\otimes W_x \longrightarrow
     K_{yxy\inv ,y}\otimes W_{yxy\inv }\otimes  \CC_y
  \end{equation*}
which, by \eqref{eq:6}, is the desired isomorphism~\eqref{eq:24}. 
        \end{proof}

  \section{2-cocycles on tori}\label{sec:6}

 \subsection{$H^4(BG; \ZZ)$ for torus groups}\label{subsec:3.1}

 Let $G=T$~be a compact connected abelian Lie group, i.e., a torus.
Associated to it are the dual lattices (finitely generated free abelian
groups)
  \begin{equation*}
     \begin{aligned} \Pi &= \Hom(\TT,T) \cong H_1(T; \ZZ) \cong H_2(BT; \ZZ) \\ \Lambda
      &=\Hom(T,\TT) \cong H^1(T; \ZZ)\cong H^2(BT; \ZZ)\end{aligned} 
  \end{equation*}
Let $\mathfrak{t}$~be the Lie algebra of~$T$.  Then $\Pi \subset
\mathfrak{t}$ by differentiation of a homomorphism, and dually~$\Lambda
\subset \mathfrak{t}^*$, also by differentiation.  The cohomology ring
of~$BT$ is the symmetric ring on~$H^2(BT; \ZZ)=\Lambda$, so in particular 
  \begin{equation}\label{eq:14}
      H^4(BT; \ZZ)\cong \Sym^2\Lambda 
  \end{equation}
We identify~$\Sym^2\Lambda $ with the group of homogeneous quadratic
functions $q\:\Pi \to\ZZ$.  For there is a natural quotient map $\Lambda
^{\otimes 2}\to \Sym^2\Lambda $ with kernel the alternating tensors, and the
value~$q(\pi )$ of the lift of an element of~$\Sym^2\Lambda $ on~$\pi \otimes
\pi $ is independent of the lift.  Then the symmetric bi-additive
homomorphism (form)
  \begin{equation}\label{eq:9}
     \langle \pi _1,\pi _2 \rangle = q(\pi _1+\pi _2)-q(\pi _1)-q(\pi
     _2),\qquad \pi _1,\pi _2\in \Pi , 
  \end{equation}
is even ($\langle \pi ,\pi \rangle\in 2\ZZ$) and $q(\pi )=\frac 12\langle \pi
,\pi \rangle$.  Therefore, the group~$\Sym^2\Lambda $ in~\eqref{eq:14} is
also isomorphic to the group of even forms $\Pi \times \Pi \to\ZZ$.

        \begin{definition}[]\label{thm:9}
 A class in~$H^4(BT; \ZZ)$ is \emph{nondegenerate} if the corresponding
form~$\langle -,- \rangle$ is nondegenerate over~$\QQ$.
        \end{definition}

\noindent
 The form induces a homomorphism $\tau \:\Pi \to\Lambda $; it is
nondegenerate over~$\QQ$ if the map $\tau _{\QQ}\:\Pi \otimes \QQ\to\Lambda
\otimes \QQ$ on rational vector spaces is an isomorphism, or equivalently if
$\tau $~is injective.

        \begin{remark}[]\label{thm:8}
 The group~$(\Sym^2\Pi )^*$ of all symmetric bi-additive homomorphisms $\Pi
\times \Pi \to\ZZ$ fits into the exact sequence 
  \begin{equation*}
     0\longrightarrow \Sym^2\Lambda \longrightarrow (\Sym^2\Pi
     )^*\longrightarrow \Hom(\Pi ,\zt)\longrightarrow 0 
  \end{equation*}
where the quotient map takes a form~$\langle -,- \rangle$ to $\pi \mapsto
\langle \pi ,\pi \rangle\pmod2$.  The quotient is the cohomology
group~$H^2(BT;\zt)$, the kernel is the cohomology group~$H^4(BT;\ZZ)$, and we
can identify the middle term with the cohomology group~$h^4(BT)$, where
$h$~is the first cohomology theory mentioned in Remark~\ref{thm:4}.
        \end{remark}

Fix a \emph{nondegenerate} class in~$H^4(BT; \ZZ)$ with corresponding
form~$\langle -,- \rangle$ and homomorphism $\tau \:\Pi \to\Lambda $.
Applying~$\otimes\, \RR$ we extend the form to~$\mathfrak{t}\times
\mathfrak{t}$ and obtain a linear map $\tR \:\mathfrak{t}\to\mathfrak{t}^*$.
The nondegeneracy implies that $\tR $~is invertible.  Let $L\to
\mathfrak{t}\times \mathfrak{t}$ be the trivial line bundle and lift the
action of~$\Pi \times \Pi $ on~$\mathfrak{t}\times \mathfrak{t}$ to~$L$ by
setting
  \begin{equation}\label{eq:15}
     (\pi ,\pi ')\:L_{\xi ,\xi '}\longrightarrow L_{\xi +\pi ,\xi '+\pi
     '},\qquad \pi ,\pi '\in \Pi ,\quad \xi ,\xi '\in \mathfrak{t},
  \end{equation}
to act as multiplication by 
  \begin{equation}\label{eq:16}
     e\Bigl(\frac {\langle \pi ,\xi ' \rangle - \langle \xi ,\pi '
     \rangle + \langle \pi ,\pi ' \rangle}2 \Bigr) ,
  \end{equation}
where $e(a)=e^{2\pi ia}$ for a real number~$a$.  Also, define the
correspondence 
  \begin{equation}\label{eq:17}
     \xymatrix@!C{&C\ar[dl]_{p_1}\ar[dr]^{p_2} \\ T && \Lambda /\tau (\Pi
      )\ar@{-->}@<1ex>[ul]^s} 
  \end{equation}
by $C=(\mathfrak{t}\oplus \Lambda )/\Pi $ where the action by~$\pi \in \Pi $
on~$(\xi ,\lambda )\in \mathfrak{t}\times \Lambda $ is 
  \begin{equation}\label{eq:18}
     \pi \cdot (\xi ,\lambda ) = \bigl(\, \xi +\pi ,\lambda +\tau (\pi )\,
     \bigr) .
  \end{equation}

        \begin{proposition}[]\label{thm:12}
 
 \begin{enumerate}[(i)]
 \item The expression~\eqref{eq:16} does lift the action of~$\Pi \times \Pi
$, whence there is a quotient hermitian line bundle $L\to T\times T$.  There
are natural isomorphisms
  \begin{equation}\label{eq:19}
     L_{x,y'}\otimes L_{x,y}\longrightarrow L_{x,y'y},\qquad x,y,y'\in T, 
  \end{equation}
which lift to the identity map on~$\mathfrak{t}^{\times 3}$ and which satisfy
an associativity condition.  Hence for each~$x\in T$ the bundle $L_{x,-}\to
T$ determines a central extension $\TT\to \hT_x\to T$.
 \item The fiber of~$p_1$ over~$x\in T$ may be identified with the $\Lambda
$-torsor of splittings of $\TT\to \hT_x\to T$.  The splitting~$\chi \mstrut
_{(\xi ,\lambda )}$ corresponding to~$(\xi ,\lambda )\in \mathfrak{t}\times
\Lambda $ is determined by
  \begin{equation}\label{eq:20}
     \chi \mstrut _{(\xi ,\lambda )}(\xi ') = e\Bigl(\bigl\langle\, \tRi
     (\lambda)-\frac \xi 2\,,\,\xi ' \,\bigr\rangle \Bigr)\;\in L_{\xi ,\xi
     '} ,\qquad \xi '\in \mathfrak{t}. 
  \end{equation}
 \item $C$~is a group and $p_2$~is split by the homomorphism $s\:\Lambda /\tau
(\Pi )\to C$ defined by
  \begin{equation}\label{eq:21}
     s(\lambda ) = \bigl(\tRi (\lambda ),\lambda \bigr)\;\in
     \mathfrak{t}\oplus  \Lambda ,\qquad \lambda \in \Lambda .
  \end{equation}
 \end{enumerate}
        \end{proposition}

\noindent
 The proof is a series of straightforward verifications from~\eqref{eq:16},
\eqref{eq:18}, \eqref{eq:20}, and~\eqref{eq:21}. 

        \begin{notation}[]\label{thm:29}
 Let $\hF\subset C$~denote the image of~$s$ and $F\subset T$ the
image~$p_1(\hF)$. 
        \end{notation}

\noindent
 Both~$F$ and~$\hF$ are finite groups isomorphic to~$\Lambda /\tau (\Pi )$.
The abelian groups~$F$ and~$\Lambda /\tau (\Pi )$ are in Pontrjagin duality
by pairing characters in~$\Lambda $ with elements of~$F\subset T$.
The map~$p_2$ has as fibers affine spaces for~$\mathfrak{t}$.  The
subgroup~$\hF$ contains a unique element in each affine space.  Also, for
each~$x\in F$ the lift ~$\hat x\in \hF\subset C$ defines a distinguished
projective character of~$\hT_x$.  Observe that $\mathfrak{t}\subset C$ as the
fiber of~$p_2$ over~$0$, and $C$~is isomorphic to the direct
sum~$\mathfrak{t}\oplus \hF$. 

A class in~$H^4(BT)$ may also be represented by a 2-cocycle\footnote{ A
2-cocycle on~$T$ with values in hermitian lines is defined as
in~Definition~\ref{thm:11} with the additional requirement that the line
bundle $K\to T\times T$ and isometry~\eqref{eq:4} be smooth.}, as in
Proposition~\ref{thm:7}, and it has an explicit construction analogous to that
of~$L$ in~\eqref{eq:15}.  It depends on a choice of a (nonsymmetric) bilinear
form $B\:\Pi \times \Pi \to\ZZ$ which whose symmetrization is the
form~\eqref{eq:9}:
  \begin{equation*}
     B(\pi _1,\pi _2) + B(\pi _2,\pi _1 )=\langle \pi _1 ,\pi _2
     \rangle,\qquad \pi  _1,\pi _2\in \Pi . 
  \end{equation*}
Namely, let $K\to\mathfrak{t}\times \mathfrak{t}$
be the trivial line bundle and lift the action of~$\Pi \times \Pi $
on~$\mathfrak{t}\times \mathfrak{t}$ to~$K$ by setting 
  \begin{equation}\label{eq:67}
     (\pi ,\pi ')\:K_{\xi ,\xi '}\longrightarrow K_{\xi +\pi ,\xi '+\pi
     '},\qquad \pi ,\pi '\in \Pi ,\quad \xi ,\xi '\in \mathfrak{t}, 
  \end{equation}
to act as multiplication by  
  \begin{equation}\label{eq:68}
     e\Bigl(\frac {B( \pi ,\xi ' ) - B( \xi ,\pi ' ) +
     B( \pi ,\pi ' )}2 \Bigr) , 
  \end{equation}

        \begin{proposition}[]\label{thm:10}
 \begin{enumerate}[(i)]
 \item The expression~\eqref{eq:68} does lift the action of~$\Pi \times \Pi
$, whence there is a quotient hermitian line bundle $K\to T\times T$.  There
are natural isomorphisms~$\theta \mstrut _{x,y,z}$ as in~\eqref{eq:4} which
lift to the identity map on~$\mathfrak{t}^{\times 3}$ and which satisfy the
cocycle condition~\eqref{eq:11}.
 \item There is an isomorphism 
  \begin{equation*}
     L_{x,y} \xrightarrow{\cong } K\inv _{x,y}\otimes K_{y,x} 
  \end{equation*}
such that \eqref{eq:19}~is $\theta \mstrut _{x,y',y}\theta \mstrut
_{y',y,x}\theta \inv _{y',x,y}$.
 \end{enumerate} 
        \end{proposition} 

\noindent
 The proof is a series of straightforward verifications.

 \subsection{Classical descriptions}\label{subsec:6.1}

We continue with the notation of~\ref{subsec:3.1}. Recall that the level is a
class in~$H^4(BT;\ZZ)$ and is represented by a homogeneous quadratic map
$q\:\Pi \to\ZZ $.  It determines a homomorphism $\tau \:\Pi \to\Lambda $.
The finite subgroup~$F\subset T$, defined in Notation~\ref{thm:29}, may be
identified as 
  \begin{equation*}
     F\cong \tau _{\QQ}\inv (\Lambda )/\Pi \subset (\Pi \otimes \QQ)/\Pi.      
  \end{equation*}
Then $q$~induces a homogeneous quadratic map $\qQZ\:F\to\QZ$.  We give a
topological interpretation.  For an abelian group~$A$ and nonnegative
integer~$n$, let $K(A,n)$~be the corresponding Eilenberg-MacLane space.

        \begin{lemma}[{\cite[Theorem~26.1]{EL}}]\label{thm:30}
 Let $A,B$ be discrete abelian groups.  Then the set of homogeneous quadratic
forms $q\:A\to B$ is isomorphic to the homotopy classes of maps
$\tq\:K(A,2)\to K(B,4)$.
        \end{lemma}

\noindent
 Applying the lemma we may represent~$\qQZ$ by a continuous map
  \begin{equation}\label{eq:49}
     \tqQZ\:K(F,2)\longrightarrow K(\QZ,4) .
  \end{equation}
 
Now for any manifold~$X$ let $\F_F(X)$~denote the 2-groupoid of
``$F$-gerbes'' on~$X$.  Recall that an $F$-gerbe\footnote{Physicists are
familiar with gerbes as fields in a field theory in string theories, where
they are known as ``$B$-fields''.} is a geometric representative of a class
in~$H^2(X;F)$.  One possible topological model is that an $F$-gerbe on~$X$ is
a map $X\to K(F,2)$.  In this model $\F_F(X)$~is the fundamental 2-groupoid
of the mapping space~$\Map\bigl(X,K(F,2) \bigr)$.  There are three nonzero
homotopy groups: 
  \begin{equation*}
     \pi _0\F_F(X)\cong H^2(X;F),\qquad \pi _1\F_F(X)\cong H^1(X;F),\qquad
     \pi _2\F_F(X)\cong H^0(X;F). 
  \end{equation*}
Composition with~\eqref{eq:49} gives the lagrangian of the field theory.
Suppose $X$~is a closed oriented 4-manifold.  Then the action is defined by
integrating the lagrangian, which in this case means pairing with the
fundamental class given by the orientation.  The result only depends on the
equivalence class of the gerbe and gives a homogeneous quadratic map
  \begin{equation*}
     q\mstrut _X\:H^2(X;F)\longrightarrow \QZ ,
  \end{equation*}
which is defined using the cup square and the quadratic form~$q$.  The
quantum invariant is then a finite path integral~(\S\ref{sec:5}) over the
stack of~$F$-gerbes:
  \begin{equation}\label{eq:52}
     \A\mstrut _F(X) = \sum\limits_{\G} \frac{\# H^0(X;F)}{\#H^1(X;F)} \;e^{2\pi
     iq\mstrut _X(\G)}, 
  \end{equation}
where the sum is over a set of representative F-gerbes on~$X$.
(Compare~\eqref{eq:31}.)  This Gauss sum may be evaluated
explicitly~\cite{HM}:
  \begin{equation}\label{eq:53}
     \begin{split} \A\mstrut _F(X) &= \frac{\# H^0(X;F)}{\#H^1(X;F)} \;
      \sqrt{\#H^2(X;F)}\;\exp\Bigl[2\pi i(\sign \form)(\sign X)/8\Bigr]  \\
       &= (\sqrt{\#F})^{\Euler X}\,\mu ^{(\sign\form )(\sign X)} 
 ,\end{split} 
  \end{equation}
where $\sign X$~is the signature and $\Euler X$ the Euler characteristic of
the 4-manifold~$X$; $\sign\form$~is the signature of the bilinear
form~$\form=\langle -,- \rangle$ in~\eqref{eq:9} associated to~$q$ (after
tensoring with~$\QQ$); and $\mu =\exp(2\pi i/8)$ is a primitive
$8^{\textnormal{th}}$~root of unity.

We claim that the finite path integral procedure of~\S\ref{sec:5}, applied to
the $F$-gerbes and the quadratic form~$q$, defines an invertible
4-dimensional TQFT~$\A_{F}\mstrut $.  For this we need to first specify a
target symmetric monoidal 4-category~$\C$. The relevant category can be
described informally as a ``de-looping'' of the $3$-category of Definition
\ref{canst}. Roughly speaking, the objects of $\C$ are braided tensor
categories $\A$ (which we can think of as {\it associative algebras} in the
setting of tensor categories).  A $1$-morphism from $\A$ to $\A'$ is an
$\A$-$\A'$ bimodule in the setting of tensor categories: that is, a tensor
category $\D$ equipped with commuting central actions of the braided monoidal
categories $\A$ and $\A'$. The $2$-morphisms in $\C$ are given by linear
categories, the $3$-morphisms by functors, and the $4$-morphisms by natural
transformations.  (See~\ref{cs7.2} for further discussion.)

Here we give a second classical description of the anomaly theory which leads
to a finite dimensional but not finite path integral, and we use some
heuristics in its evaluation on a 4-manifold.  In this theory the finite
group~$F$ is replaced by the Lie algebra~$\mathfrak{t}$ with the discrete
topology.  Now we use the real homogeneous quadratic form $q\mstrut
_{\RR}\:\mathfrak{t}\to\RR$ and apply Lemma~\ref{thm:30} to obtain
  \begin{equation}\label{eq:54}
     K(\mathfrak{t},2)\longrightarrow K(\RR,4)\longrightarrow K(\RZ,4) .
  \end{equation}
Let $\Ft(X)$~be the 2-groupoid of $\mathfrak{t}$-gerbes on a closed oriented
4-manifold~$X$. Then \eqref{eq:54}~determines a homogeneous quadratic map
  \begin{equation}\label{eq:55}
     q\mstrut _{X}\:H^2(X;\mathfrak{t})\longrightarrow \RZ. 
  \end{equation}
The finite sum~\eqref{eq:52} is now replaced by a sum over an uncountable
set, which we interpret as an integral:
  \begin{equation*}
     \At(X) = \int\limits_{\G\in H^2(X;\mathfrak{t})} \frac{\vol
     H^0(X;\mathfrak{t})}{\vol H^1(X;\mathfrak{t})} \;e^{2\pi iq\mstrut _X(\G)}.
  \end{equation*}
Here `$\vol$'~denotes a formal volume which we regularize below.  This is a
Gaussian integral, and we evaluate it as\footnote{If $Q$~is a symmetric bilinear
form on a finite dimensional real vector space~$V$, it induces a map $V\to
V^*$ whose determinant is a map $\Det Q\:\Det V\to\Det V^*$, so an
element~$\Det Q\in( \Det V^*)^{\otimes 2}$.  The integral of~$e^{iQ(x,x)/2}$
over~$V$ has an algebraic evaluation as
  \begin{equation*}
     \frac{e^{2\pi i(\sign Q)/8}}{\sqrt{|\det Q|}}\quad \in |\Det V|, 
  \end{equation*}
where $|\Det V|$~is the real line associated to~$V$ by the absolute value
character of~$\RR^{\not= 0}$.  A translation-invariant volume form on~$V$ may
be viewed as an element of the dual real line~$|\Det V^*|$, which then gives
a numerical answer which matches the usual Gaussian integral.  This explains
the signature factors in~\eqref{eq:57}.}
  \begin{equation}\label{eq:57}
     \begin{split} \At(X) &= \frac{\vol  H^0(X;\mathfrak{t})}{\vol
     H^1(X;\mathfrak{t})} \; 
      \sqrt{\vol H^2(X;\mathfrak{t})}\;\exp\Bigl[2\pi i(\sign \form)(\sign
     X)/8\Bigr] 
     \\ &=  \lambda ^{\Euler X}\,\mu ^{(\sign\form )(\sign X)}, \end{split} 
  \end{equation}
where $\lambda $~is a constant we choose equal to~$\sqrt{\#F}$ to
match~\eqref{eq:53}. 

Before leaving these classical descriptions we indicate a classical coupling
of the usual toral Chern-Simons to the classical gerbe theory~$\At$ on a
compact oriented 4-manifold~$X$ with boundary a closed oriented
3-manifold~$Y$.  We freely use generalized differential
cohomology~\cite{HS}.\footnote{We use differential theories based on the
Eilenberg-MacLane spectrum~$H\Pi $.}  Let $\G$~be a $\mathfrak{t}$-gerbe
on~$X$; since $\mathfrak{t}$~has the discrete topology $\G$~ is flat.  Its
restriction~$\partial \G$ to~$Y$ has an exponential~$\exp\partial \G$ which
is a topologically trivial flat $T$-gerbe.  The field in toral Chern-Simons
is a `non-flat trivialization'~$P$ of~$\exp\partial \G$.  More precisely,
there is a groupoid whose objects are $\mathfrak{t}$-gerbes on~$Y$ and whose
morphisms are torsors for the differential cohomology group~$\cHP^2(Y)$;
these morphisms are equivalence classes of non-flat trivializations of the
exponentials of the $\mathfrak{t}$-gerbes.  (The actual non-flat
trivializations are morphisms in a 2-groupoid.)  Now $P$~extends to a
non-flat trivialization of~$\exp\G$ on~$X$.  Let $\omega \in \Omega
^2(X;\mathfrak{t})$ be its curvature.  Work over a base manifold~$S$.  Then
the quadratic form~\eqref{eq:54}, applied to~$\partial \G$ and integrated
over the fibers of $Y\to S$, yields a flat $\RZ$-bundle over~$S$ whose
equivalence class in~$H^1(S;\RZ)$ is computed by a relative version
of~\eqref{eq:55} for the relative 3-manifold~$Y\to S$.  Because $\partial
\G$~is extended to a $\mathfrak{t}$-gerbe on~$X$, this $\RZ$-bundle comes
with a trivialization.  The Chern-Simons action is a section of this circle
bundle, and using the trivialization may be identified with the function
  \begin{equation}\label{eq:59}
     \int_{X/Z}\langle \omega \wedge \omega \rangle\quad
     (\textnormal{mod}\;\ZZ)
  \end{equation}
on~$S$. 
 
In the quantum Chern-Simons theory we integrate the exponential
of~\eqref{eq:59} over the stack of non-flat trivializations~$P$ for
fixed~$\G$.  The result lives in the complex line bundle $L(\partial \G)\to
S$ which is the exponential of the $\RZ$-bundle in the previous paragraph.
Automorphisms of~$\partial \G$ act on~$L(\partial \G)$ and the result of the
path integral is invariant.  Therefore, if these automorphism act
nontrivially the path integral vanishes.  (This is called the `Gauss law' in
physics.)  Now an automorphism~$\alpha $ of~$\partial \G$ acts through its
equivalence class ~$[\alpha ]\in H^1(Y;\mathfrak{t})$, and the action only
depends on the equivalence class ~$[\partial \G]\in H^2(Y;\mathfrak{t})$ of
the gerbe: namely, it acts as multiplication by~$\exp\bigl(2\pi i\,\langle
[\alpha ]\smile[\partial \G] \rangle\,\bigr)$.  This shows that the path
integral vanishes unless~$\partial \G$ is trivializable as a flat
$\mathfrak{t}$-gerbe.  Relative to a trivialization the Chern-Simons
field~$P$ is a usual $T$-bundle with connection and we recover the standard
description of classical Chern-Simons.

\section{The basic tensor category and its center}\label{sec:cs6}

\subsection{Drinfeld centers}\label{cs6.1}

Let $T$ be a torus, $\tau\in H^4(BT;\bZ)$ a non-degenerate twisting.  Recall
from Proposition~\ref{thm:10} the line bundle $K\to T\times T$, which is a
2-cocycle.  Define a convolution on the category $\sky[T]$ of sky-scraper
sheaves of finite-dimensional vector spaces on $T$, with finite support, by
setting
 \[
\bC_x *\bC_y = K_{x,y}\otimes\bC_{xy}
 \]
 for the skyscrapers at $x,y,xy\in T$. The cocycle property of $K$ ensures
that this defines a tensor category, which we denote by $\sky^\tau[T]$.  This
is an analogue of the ``twisted group ring" $\vect^\tau[G]$ discussed
in~\ref{subsec:1.5} for finite groups $G$. 

The tensor structure lifts to the universal cover $\frt$ of $T$, but it is
trivializable there. This is because we can trivialize the pullback of~ $K$
as a $2$-cocycle valued in $\pic$, compatibly with our chosen trivialization
of the pullback of $L$, as in~\eqref{eq:67} and~\eqref{eq:15}.  However, we
shall see that $\sky^\tau[\frt]$ carries a higher structure, a
\emph{braiding}. This is specified by a family of automorphism $\bC_\xi
*\bC_{\xi'} \xrightarrow{\sim} \bC_{\xi'} *\bC_\xi,\;\xi ,\xi '\in
\mathfrak{t}$ forming a bi-multiplicative section of the line bundle $L$
of~\eqref{eq:15}.  In the defining trivialization, the distinguished section
is the function
  $$ \sigma (\xi,\xi') = \exp \bigl\{-\pi{i}\langle \xi,\xi'\rangle \bigr\}.$$
The structure extends in fact to the larger category $\sky^\tau[C]$ of
sky-scrapers on the correspondence space $C= (\frt\times\Lambda)/\Pi$
in~\eqref{eq:17} by a reformulation of Proposition~\ref{thm:12}(ii).
Alternatively, there is a natural quadratic function $\theta: C\to\bT$
(Remark~\ref{selfdual}(ii) below), which determines the braided tensor
structure on $\sky^\tau[C]$ by a general construction (cf.~the end of
\ref{untwisted}).  The conceptual meaning for this structure is given by
the following.

 \begin{proposition}\label{centerT}
 \begin{trivlist}\itemsep0ex
 \item(i) The braided tensor category $\sky^\tau[C]$ is the ``continuous" 
Drinfeld center of $\sky^\tau[T]$. The natural functor from 
$\sky^\tau[C]$ to $\sky^\tau[T]$ is induced by projection. 
 \item (ii) As braided tensor categories,  $\sky^\tau[C] \equiv 
\sky^\tau[\frt]\otimes\sky^\tau[\hat{F}]$, sitting in $C$ by the 
obvious inclusions, and lifting the splitting $C=\frt\times\hat{F}$ 
of abelian groups. Moreover, the two factors are mutual commutants.
 \end{trivlist}
 \end{proposition}
 \noindent
 The last statement means that any $\sky^\tau[C]$-object braiding trivially
with all of $\sky^\tau[\hat{F}]$ is in $\sky^\tau[\frt]$, and similarly with
$\hat{F}$ and $\frt$ interchanged.

Regarding the notion of ``continuous Drinfeld center", we will limit
ourselves to the following observation. All our categories are semi-simple,
with the simple isomorphism classes corresponding to the points on the
underlying spaces. They are `categorifications' of the underlying abelian
groups.  We then ask that the half-braiding in Definition~\ref{thm:14} should
be continuous on irreducible objects, in their natural topology. Developing
this in more detail would be a distraction here for two reasons: first, the
key ingredient for us is not quite the Drinfeld center, which is a Hochschild
cohomology, but its dual notion, a Hochschild homology, which is a tensor
product. Second, as we briefly indicate in the next section, the easiest way
to justify our story rigorously is via approximations by \emph {finite
abelian groups}, in the time-tested ``lattice approximation" of quantum field
theory; and continuity plays no r\^{o}le in that setting.

\begin{proof}[Sketch of proof.]  

An alternative description of $C$ will be useful. Note that $\tau$ defines an
isogeny $\tau_\bT: T\to{T}^*$, the Langlands dual torus. Its kernel is the
group $F$ described in Notation~\ref{thm:29}.  Interpreting $T^*$ as the
moduli space of flat line bundles on $T$, the map~$\tau_{\bT}$ classifies the line
bundle $L$. Then,
 \[
C \cong \frt^*\underset{T^*}{\times}T. 
 \]
To identify this fiber product over $T^*$ with the standard description 
$(\frt\oplus\Lambda)/\Pi$, send $(\xi,\lambda)$ in the latter space 
to $(d\tau(\xi)-\lambda,e^\xi)$ in the former. 

In the second description of $C$, we interpret a point $x\in{T}$ as the
sky-scraper object $\bC_x$ and an element in the fiber of $\tau_{\bT}(x)$ as
a continuous character of the central extension of the group $T$ defined by
$L_{x,-}$; see Proposition~ \ref{thm:12}(ii). But this is precisely a
continuous half-braiding of $\bC_x$ with the simple objects of
$\sky^\tau[T]$. It is easy to see a priori that simple objects of the center
must be supported at single points, and then that they have rank one, so we
have just found all of them.

If we start at the identity with the trivial braiding, and keep a continuous 
choice of character as we move in $\frt$, we sweep out the identity copy 
of $\frt$ in $C$. The projective characters of $T$ thus swept out have 
the property that they are trivial on $F\subset T$ (where every central 
extension in the family is naturally trivialized).
On the other hand, $L_{x,-}$ has trivial holonomy at the 
points $x\in F$, and there we can choose the trivial character in $\frt^*$. 
This defines the copy $\hat{F}$ of $F$ in $C$. Clearly, this braiding \emph
{commutes} with the braiding by the copy of $\frt$ just described. 

Non-degeneracy of the quadratic form on $\frt$ implies that its commutant 
is no larger than $\sky^\tau[F]$. That the two in fact commute, and the 
analogous statement for $F$, follow from the formula for the ribbon element 
below. One can also argue directly that the braiding is defined by the 
perfect bi-character on $C$ described in Remark~\ref{selfdual}.
\end{proof}

\begin{remark}\label{selfdual}
\begin{trivlist}\itemsep0ex
\item (i) The two descriptions of $C$ make clear its remarkable property 
of being Pontrjagin self-dual: the groups $\frt^*\times{T}$ and 
$\frt\times \Lambda$ are Pontrjagin dual, and the kernel of the addition 
map of the first group to $T^*$ is dual to the quotient of the second by 
the dual inclusion of $\Pi$. This self-duality is symmetric, and is 
induced by the quadratic ``ribbon" map $\theta:C \to \bT$ below.
\item (ii) The braided tensor category $\sky^\tau[C]$ and its sub-categories
$\sky^\tau[\hat{F}], \sky^\tau[\frt]$ are in fact \emph{ribbon categories} 
with ribbon function on the simple object $(\xi,\lambda)$
 \[
\theta(\xi,\lambda) =  
\exp\,\pi{i} \bigl\{\|\tau^{-1}(\lambda)\|^2 - \|\xi-\tau^{-1}(\lambda)\|^2 
\bigr\},
 \]
 using the norm associated to the quadratic form~\eqref{eq:9}.  The square of
the braiding is given by the standard formula $\theta(XY)\theta^{-1}
(X)\theta^{-1}(Y)$.

\item (iii) The equivalence in \ref{centerT}(ii) is \emph{not} one of
\emph{ribbon} categories, because the ribbon function is quadratic and not
linear on objects.

\item (iv) One should mind that the tensor structure defined by the
restriction~ $\tau\in H^4(BF;\bZ)$ is not necessarily trivial, that is,
$\sky^\tau[F]$ may differ from $\sky[F]$ as a tensor category. For instance,
this always happens when $T=S^1$. But $\tau$ is always $2$-torsion on $F$.

\item (v) The category $\sky^\tau[\hF]$ is in fact a \emph{modular tensor
category} \cite{BK, T} and defines the 1-2-3-dimensional Chern-Simons theory
(of framed manifolds) associated to the torus $T$ at level $\tau$ via the
Reshetikhin-Turaev theorem.
 \end{trivlist}\end{remark}

\subsection{Finite approximation of $T$ and $\frt$}\label{subsec:c6.2}
Let us now describe the finite (``lattice'') approximations of $C, 
\frt, T$ and develop the finite version of Chern-Simons theory in the 
next section. Let $n$ be positive integer, which will become infinitely 
divisible in the limit\footnote{That is, we'll take the limit over $\mathbb{N}$ 
ordered by divisibility.}; denote by $T^*_{(n)}\subset T^*$ the subgroup 
of $n$-torsion points and by $T^{(n)}\to T$ the dual covering torus 
of $T$, with Galois group $\Pi^{(n)}\cong \Pi/n\Pi$, Pontrjagin dual to
$T^*_{(n)}$. 
Finally, let $T_{(nF)}$ denote the inverse image of $T^*_{(n)}$ 
in $T$ under $\tau_\bT$ and $\frt^{(n)}$ that of $T_{(nF)}$ in $T^{(n)}$. 
There is also the Pontrjagin dual $\frt^{*(nF)}$ of $\frt^{(n)}$, which 
is a Galois cover of $T^*_{(n)}$ (restricted from $T^*$) with group  
$\Lambda^{(nF)}$, Pontrjagin dual to $T_{(nF)}$. It is an exercise to 
check that the restriction $\tau_\bT: T_{(nF)} \to T^*_{(n)}$ lifts naturally
to an isomorphism $\frt^{(n)} \cong \frt^{*(nF)}$.

As $n\to\infty$, the finite group $T_{(n)}$ will be our approximation 
for $T$, $\frt^{(n)}$ will approximate the Lie algebra, $\frt^{*(nF)}$ its 
dual, $T^*_{(n)}$ will play the r\^ole of $T^*$, $\Pi^{(n)}$ that of $\Pi$ 
and $\Lambda^{(nF)}$ that of $\Lambda$.  
\begin{remark}
If $d$ is the dimension of $T$, then 
$\#T^*_{(n)} = n^d$, $\#T_{(nF)} = \#F\cdot n^d$, $\#\frt^{(n)} = \#F\cdot
{n}^{2d}$. 
\end{remark}

The twisting $\tau$ restricts to $H^4(BT_{(n)};\bZ)$ and defines a tensor
category $\mathrm{Vect}^\tau(T_{(n)})$ as discussed in ~\ref{subsec:1.5},
which is a full subcategory of $\sky^\tau[T]$. Consider
 \[
C^{(n)} := \left(\frt^{(n)} \oplus \Lambda^{(nF)}\right)/\Pi^{(n)} 
 \cong \frt^{*(nF)} \underset{T^*_{(n)}} \times T_{(nF)}.
 \]
 It is easy to check that the ribbon function $\theta$ of Remark~\ref
{selfdual}(ii) descends to a non-degenerate quadratic function on $C^{(n)}$,
and gives a braided tensor structure on $\sky^\tau[C^{(n]})$, with the
restricted twisting $\tau$). The projection $\frt^{*(nF)} \cong \frt^{(n)}
\to T_{(n)}$ defines a splitting $C^{(n)} \cong \frt^{(n)} \times \hat{F}$.
The proof of the following, discrete analogue of Proposition~\ref{centerT} is
left to the reader.  

\begin{proposition}\label{approxcenter}
\begin{trivlist}\itemsep0ex \item(i) The braided tensor category
$\sky^\tau[C^{(n)}]$ is the Drinfeld center of $\sky^\tau[T_{(nF)}]$. The
natural functor from $\sky^\tau[C^{(n)}]$ to $\sky^\tau[T_{(n)}]$ is induced
by projection.  \item (ii) As braided tensor categories, $\sky^\tau[C^{(n)}]
\equiv \sky^\tau[\frt^{(n)}]\otimes\sky^\tau[\hat{F}]$, sitting in $C^{(n)}$
by the obvious inclusions, and lifting the splitting $C^{(n)}
=\frt^{(n)}\times{F}^*$ of abelian groups. Moreover, the two factors are
mutual commutants.
\end{trivlist}
\end{proposition}

\subsection{Morita relations between our categories}\label{cs6.3}
The categories discussed in this section are closely related. Just as 
the right notion of quasi-isomorphism for algebras in their natural 
world is Morita equivalence, there is a corresponding notion for 
tensor categories and even braided tensor categories. In the next 
section, we sketch a minimal background for these notions; but let us 
now state a key result and its significance for TQFTs.  

\begin{proposition}\label{morita}
\begin{trivlist}\itemsep0ex
\item (i) $\sky^\tau[\frt]$ and $\sky^\tau[\hat{F}]$ are Morita equivalent 
by means of $\sky^\tau[T]$.
\item (ii) $\sky^\tau[\frt^{(n]})$ and $\sky^\tau[\hat{F}]$ are Morita 
equivalent by means of $\sky^\tau[T_{(n)}]$. 
\end{trivlist}
\item(iii) All these BTCs are quasi-invertible, more precisely,
$\sky^\tau[S]\otimes\sky^{-\tau}[S]$ is Morita equivalent to $\vect$ in all
cases $S=\mathfrak{t}, \hF, \frt^{(n)}$ (with $(-\tau)$ indicating the
opposite braiding).
\end{proposition}

Thus, $\sky^\tau[\frt]$ and $\sky^\tau[\frt^{(n)}]$ are quasi-invertible 
in their world. Moreover, they are equivalent to the unit braided tensor 
category $\vect$ in the following cases:
\begin{itemize}
\item When the signature of $\langle -,-\rangle$ on $\frt$ is divisible 
by $8$; 
\item When \emph{$\bZ/2$-graded} vector spaces are used.
\end{itemize}
In the first case, we apply Proposition~\ref{morita}(i) to a product $T$ of
copies of the $E_8$ maximal torus, with its generating class $\tau$;
unimodularity of the $E_8$ lattice ensures that $F=\{1\}$. The second case
uses a variant of our categories and twistings with graded vector spaces,
which we have not discussed. In this case, with $T=\mathrm{U}(1)$, the graded
tensor category $\mathrm{Gr-}\vect(\mathrm{U}(1))$ can be twisted by a class
which is \emph{half}\footnote{This half-generator is defined in the
generalized cohomology theory $h$ of Remark~\ref{thm:4}; this leads to a TQFT
for spin, rather than oriented manifolds.}
 the generator of $H^4(B\mathrm{U}(1);\bZ)$, 
and leads again to a trivial group $F$. (The group $F$ corresponding to 
the generator of $H^4(B\mathrm{U}(1);\bZ)$ has two elements.) 

The braided tensor categories above lead to \emph{invertible} $4$-dimensional
TQFTs for oriented manifolds: these are the theories described in
\ref{subsec:6.1}.  Invertibility has the obvious meaning, in the tensor
structure on TQFTs, and follows from Proposition~\ref{morita}(iii). Graded
vector spaces, however, require spin structures on the manifold. We conclude
that the (isomorphic) theories defined by $\sky^\tau[\frt]$ and $\sky^\tau
[\frt^{(n)}]$ are trivial on oriented manifolds when the signature of $\tau$
is a multiple of $8$, and always trivial on spin manifolds.\footnote {Some
structure has been swept under the carpet here. Defining the theory for
oriented or spin manifolds, rather than framed ones, requires us to specify
an action of $\mathrm{SO}(4)$ or $\mathrm{Spin}(4)$, which adds some
parameters. One of those is the constant $\lambda$ coupled to the Euler class
in~\eqref{eq:57}.}

\section{Higher algebra: tensor and bi-module categories}\label{sec:c1}

\subsection{$m$-Algebras} \label{lowdim}

Our picture of Chern-Simons theory requires an ascent to categorical 
altitudes which exceed the safe limit without special equipment: we 
will need linear $3$-categories describing (very simple) 
$4$-dimensional theories. Fortunately, there is a recursive procedure 
to produce the higher categorical objects needed and, even more fortunately,
up to the range of dimensions we study there are user-friendly models for 
these. We hope to return to an extensive discussion of these structures 
elsewhere; here we sketch their basic features.

Recall first that vector spaces and linear maps form a symmetric 
tensor category, and that algebra objects in a symmetric tensor 
category form in turn a symmetric tensor category (of one level 
higher, but who's counting?).

\begin{definition}
 A \emph{$0$-algebra} is a complex vector space, and a morphism of
$0$-algebras is a linear map.  For $m> 0$, an $m$-algebra is an algebra
object in the (symmetric tensor $m$-) category of $(m-1)$-algebras.  Morphism
between $m$-algebras are bi-module objects in the category of
$(m-1)$-algebras. (Higher morphisms are defined recursively.)
\end{definition}

 The $2$-category of ($1$-)algebras, modules and intertwiners is familiar
enough---see~\ref{alg}---and here we just give analogous pictures (or
adequate substitutes) for the next two levels. A $2$-algebra is an algebra
$A$ together with an $(A,A^{\otimes 2})$-bi-module $M$ defining the
$2$-multiplication and a left $A$-module $E$ defining the identity; there
must be the associativity and unit intertwiners ($2$-morphisms in the
category of algebras) satisfying the `obvious' compatibility rules. A
morphism of $2$-algebras, from $A$ to $B$, is a $(B,A)$ bi-module in the
world of algebras: an algebra~$N$, plus an $(N,B\otimes{N})$-bi-module $P$
and an $(N,N\otimes{A})$-bi-module $Q$, plus compatibility isomorphisms
defining the bi-algebra structure.

In general, a $2$-algebra structure on $A$ defines a tensor structure on the
category of left $A$-modules; sending $A \mapsto (A\mathrm{-Mod})$ is a fully
faithful functor from the $3$-category of $2$-algebras to that of tensor
categories. Modulo the problem of realizing linear categories as module
categories, one can use tensor categories instead of $2$-algebras. This, of
course, is what we have been doing.

For example, any commutative algebra becomes a $2$-algebra via its own
multiplication map; a morphism between commutative algebras $A$ and $B$,
viewed as $2$-algebras, is simply a $B\otimes A$-algebra.  However,
\emph{commutative} algebras have a more obvious and economical embedding into
tensor categories, sending $A$ to the tensor category with one object and
endomorphism ring $A$.  We will write $A[2]$ when thinking of $A$ as a
$2$-algebra.

More relevant is the example of a Hopf algebra $H$, which becomes a
$2$-algebra by means of the multiplication $M:= H\otimes{H}$, which is a left
$H$-module via the Hopf structure.  For the Hopf algebra $A=\bC(G)$ of a
finite group with the point-wise multiplication, but co-multiplication
induced by the Hopf structure, the category $(A\mathrm{-Mod})$ is that of
vector bundles over $G$; and the tensor category defined by the Hopf
$2$-algebra structure is precisely the earlier $\vect[G]$
in~\ref{subsec:1.5}. A twisting class $\tau\in H^4(BG;\bZ)$ changes the
associator of the $2$-multiplication.

One level up, a $3$-algebra structure on $A$ can be interpreted as 
follows: the $2$-algebra structure turns $(A\mathrm{-Mod})$ into a 
tensor category, having a $2$-category of modules; and the $3$-algebra 
structure gives a \emph{tensor structure} on the latter $2$-category. 
(This, in turn, has a linear $3$-category of modules, affirming the 
$3$-categorical nature of a $3$-algebra.) 
For example, if $A$ is commutative, two $(A\mathrm{-Mod})$-modules can 
be tensored over $(A\mathrm{-Mod})$ to produce a new $(A\mathrm{-Mod})
$-module, since the latter is a symmetric tensor category. Similarly, 
in the case of the group ring, $\vect[G]$-modules can be tensored over 
$\vect[G]$, using the Hopf structure on the latter. This is usually 
not possible for the twisted versions $\vect^\tau[G]$, where the Hopf 
structure is broken by the twisting; this is much like the case of 
twisted group rings over $\bC$ defined from $\bT$-central extensions 
of the group, which do not usually have $2$-algebra promotions. 

\subsection{Braided tensor categories}  \label{cs7.2}

Manipulating tensor $2$-categories can be rather daunting, and no doubt the
most general $3$-algebras are no friendlier; but in some cases they can be
captured by a more concise structure, namely a \emph{braided tensor category}
(BTC). These are special kinds of algebra objects in the category of all
tensor categories, and represent the simplest type of structure on the
category $(A\mathrm{-Mod})$ which promotes $A$ to a $3$-algebra. For example,
each braided tensor category $\sky^\tau[S]$ of the previous section promotes
the algebra of functions on $S$, with \emph{pointwise} multiplication, into a
$3$-algebra.\footnote{The cases of $S=\frt$ and $S=C$ require \emph{a priori}
a grain of salt, because we are considering very special modules, but in fact
the structure supplied does give a $3$-algebra in each case, albeit with poor
finiteness conditions.} As we will see, these higher structures are the
natural result of \emph{quantization}.

Even if we restrict to BTC among all algebra objects in the symmetric 
$3$-category of tensor categories, the description of all module 
objects---hence the description of bi-module objects, which are 
the morphisms in the $4$-category of BTCs---can be awkward; see the
case of $2$-algebras above. However, there is a nice class of special 
module objects $M$, those for which the action of $B$ on $M$ is 
defined by a functor $B\otimes{M} \to M$. This must be a \emph
{tensor functor}, which forces the action of $B$ to half-braid with 
the multiplication on $M$; in other words, the algebra map $B\to M$ 
induced by tensoring with $1\in M$ must lift to a braided tensor 
functor into the Drinfeld center $Z(M)$. We call these `$B$-modules' 
\emph{half-braided algebras} over $B$. The reader should think of 
the analogy with a commutative $1$-algebra $A$ and the $2$-category 
of central algebras over it; here, the braiding can be thought of 
as a homotopy between left and right multiplications.    

In spite of these difficulties, the following simplified four-category 
$\mathcal{C}$ will be a suitable replacement the $4$-category of $3$-algebras. 
Objects of $\mathcal{C}$ will be braided tensor categories; morphisms 
from $A$ to $B$ will be half-braided $(B,A)$ bi-algebra categories, 
that is, tensor categories $M$ with a braided tensor functor $B\otimes 
A^{op}$ to $Z(M)$. $2$-morphisms will be bi-modules between bi-algebras 
(with compatible $B\otimes {A}^{op}$-action), and then functors 
and natural transformations round this up. 

\subsection{Quasi-isomorphisms}
The key moral of the story is that the correct notion of equivalence 
becomes increasingly obscure for higher algebras, although not less precise. 
The correct notion is always a pair of functors $f,g$ with $f\circ{g}$ and 
$g\circ{f}$ both equivalent to the respective identities.\footnote{This 
is not circular, since it relies on the lower-algebra definition of 
equivalence.} Thus, for $1$-algebras we have Morita equivalence. 
For $2$-algebras, the familiar Morita conditions 
 \[
P\underset{R}\otimes Q \equiv S, \qquad Q\underset{S}\otimes P \equiv R
 \]  
 become Morita equivalences themselves, as $P$ and $Q$ are now bi-algebras.
This continues, as we now illustrate in the proof of
Proposition~\ref{morita}.  Recall for this purpose the braided tensor
categories $\sky^\tau[\hat{F}]$, $\sky^\tau[\frt]$, $\sky^\tau[C]$ from the
previous section, the latter two accompanied by their finite approximations
$\sky^\tau[\frt^{(n)}]$ and $\sky^\tau[C^{(n)}]$.  In addition, the tensor
categories $\sky^\tau[T]$ and $\sky^\tau[T_{(n)}]$ are half-braided
bi-algebras over the first three, respectively their finite versions. This
makes them into $1-$morphisms in the $4$-category $\mathcal{C}$.

\begin{proof}[Sketch of proof of Proposition \ref{morita}.]
We prove part (ii), part (i) is similar. We must produce Morita 
equivalences  
 \begin{equation}\label{morita2}
 \begin{split}
\sky^\tau[T_{(n)}]\underset{\sky^\tau[\frt^{(n)}]}\bigotimes \sky^\tau[T_{(n)}] 
\sim \sky^\tau[\hat{F}],  \\
\sky^\tau[T_{(n)}]\underset{\sky^\tau[\hat{F}]}\bigotimes \sky^\tau[T_{(n)}] 
\sim \sky^\tau[\frt^{(n)}],
 \end{split}  
 \end{equation}
and similarly for $\frt$ and $T$. The Morita objects realizing~\eqref{morita2}
are $\sky^\tau[T_{(n)}]$ in all cases (respectively, $\sky^\tau[T]$ for part
(i)). Indeed, the desired identity is
 \[
\sky^\tau[T_{(n)}]\underset{\left(\sky^\tau[T_{(n)}]
\underset{\sky^\tau[\frt^{(n)}]}\otimes \sky^\tau[T_{(n)}]\right)}
\bigotimes \sky^\tau[T_{(n)}] \equiv \sky^\tau[\hat{F}],
 \] 
 and the obvious permutations. Were we to tensor over the braided tensor
category $\sky^\tau[T_{(n)}]\otimes \sky^\tau[T_{(n)}]$ instead, the left
side would be the Hochschild homology $A\underset{A\otimes{A}}\otimes{A}$.
By semi-simplicity (or rather, self-duality of $\sky^\tau[T_{(n)}]$ as a
self-bi-module), this would give (a linear category equivalent to) the
Drinfeld center $\sky^\tau[C^{(n)}]$ of $\sky^\tau[T_{(n)}]$.  Instead, the
effect of working over the tensor product
 \[
\sky^\tau[T_{(n)}] \underset{\sky^\tau[\frt^{(n)}]}\otimes 
\sky^\tau[T_{(n)}]
 \]  
 is to pick the \emph{relative} Drinfeld center of $\sky^\tau[T_{(n)}]$ over
$\sky^\tau[\frt^{(n)}]$, the commutant of $\sky^\tau[\frt^{(n)}]$ in the full
center $\sky^\tau[C^{(n)}]$. By Proposition~\ref{approxcenter}, this is
$\sky^\tau[\hat{F}]$.  Exchanging~$\hat{F}$ and~$\frt^{n}$ and repeating the
argument we obtain the second identity in~\eqref{morita2}.

Regarding Part (iii), this is a similar argument. Noting that $\sky^{-\tau}
[S]$ can be identified with the opposite braided tensor category of 
$\sky^\tau[S]$, we will use the tensor category $\sky^\tau[S]$ with 
its left-right action on itself as Morita bi-module. Furthermore, 
we use $\sky^\tau[S]$ again to produce the Morita equivalence of
 \[
\sky^\tau[S]\underset{\sky^\tau[S]\otimes \sky^{-\tau}[S]}\bigotimes
\sky^\tau[S] 
 \] 
with $\vect$. Indeed, after identifying one copy of $\sky^\tau[S]$ with the
dual category, 
the tensor square of $\sky^\tau[S]$ over the left category is the relative 
Drinfeld center of $\sky^\tau[S]$ \emph{over $\sky^\tau[S]\otimes \sky^
{-\tau}[S]$}. However,  the center is the double $\sky^\tau[S]\otimes 
\sky^{\tau}[S^*]$, and because $\tau$ gives an isomorphism $S\cong S^*$, 
this double is nothing but $\sky^\tau[S]\otimes \sky^{-\tau}[S]$ again, 
with its half-braided tensor action on $\sky^\tau[S]$. Non-degeneracy 
of $\tau$ again ensures that the identity is the only simple object 
braiding trivially with everything; therefore the relative center is 
$\vect$, as desired.   
\end{proof}

\section{Quantization of groupoids} \label{cs8}

In this section, we outline a ``canonical quantization" procedure starting
from a ``classical'' topological quantum field theories with target higher
groupoids (satisfying some finiteness conditions).  In other words, we
elaborate on the map~$\Sum_n$ in~\eqref{eq:3.9}.  For our purposes, ``higher
groupoids" are spaces, with $0$-groupoids being discrete sets and ordinary,
or $1-$groupoids being homotopy $1-$types.  The case~$n=2$ is used
in~\ref{subsec:6.1}.  The constraint is finiteness of the homotopy groups. We
hope to develop the full story elsewhere; here we outline the construction
while flagging the cases of immediate interest.

There is a quantization procedure for each $n\ge 0$, leading to a TQFT in the
respective dimension. At its root is a linearization functor from spaces to
higher algebras. (This works without finiteness.)  This ``higher groupoid
ring" is an $m$-algebra, which is what a TQFT in dimension $n= m+1$ assigns
to a point. The construction enhances the group algebra of a finite group
($m=1$), or even more basically, of the vector space of functions on a finite
set ($m=0$).  The discrepancy between $m$ (the algebra level) and $n$ (the
TQFT dimension)---for which we apologize---is caused by the fact that the
term $n$-vector space has been used for a much more restrictive notion than
our $(n-1)$-algebras.

The $m$-algebra associated to a (connected, pointed) space $X$ has as its
category of modules the $m$-category of representations of the based loop
space $\Omega{X}$ on $(m-1)$-algebras; for disconnected spaces, we must sum
over components. The dependence on base points can be removed by viewing its
representations as the category of local systems of $(m-1)$-algebras over
$X$. We will write $R_m(X)$ for the groupoid algebra when not relying on
base-point choices.

There is an untwisted construction, leading to a TQFT of unoriented manifolds, 
and a construction with twists by a phase, which requires orientations.

        \begin{notation}[]\label{thm:33}
 Let $G$~be a group and $A$~an algebra.  Then $A[G]$~denotes the group
algebra of~$G$ under convolution and $A(G)$ the $A$-valued functions on ~$G$
under pointwise multiplication.
        \end{notation}

\subsection{Low $m$ examples}\label{untwisted}
Let X be a space, which for our application may be assumed to have 
finitely many nonzero homotopy groups, each of which is finite. We now 
produce a candidate for the $m$-algebra $R_m(X)$.  

For $m=-1$, our construction produces the number of isomorphism 
classes of points of the groupoid, weighted by their automorphisms;
\begin{equation}\label{sumpi}
\sum_{x\in\pi_0X} \prod_{i\ge 1} \left(\#\pi_i(X,x)\right)^{(-1)^i}
\end{equation}
A number counts as a ``$(-1)$-algebra"; for natural numbers at least, 
which are cardinalities of sets, this matches our intuition. The 
weight of each isomorphism class is the alternating product of the orders 
of the homotopy groups, which places an obvious finiteness condition. 
Without it, we must abandon the top layer of the TQFT. 

For $m=0$, we get the vector space of functions on~$\pi _0X$. 

For $m=1$, we produce the usual groupoid algebra. After 
a choice of base-points, this is the direct sum of the group algebras of 
the $\pi_1$'s of the components, but a Morita equivalent base-point 
free construction is the path algebra of $X$.\footnote{We are 
assuming a discrete model for spaces, such as simplicial sets.} 

Things become more interesting for $m=2$: the groupoid $X= K(\pi_2,2)$
quantizes to the $2$-algebra~$\CC[\pi _2][2]$ associated to the commutative
group algebra $\bC[\pi_2]$. If, in addition, a $\pi_1$ is present, we get a
\emph {crossed product $2$-algebra} $\pi_1\ltimes \bC[\pi_2][2]$. This
easiest to describe by means of the tensor structure on the category of
modules of the underlying $1$-algebra.  As a $1$-algebra, $\pi_1\ltimes
\bC[\pi_2][2]$ consists of the functions on $\pi_1$ with values in the
algebra $\bC[\pi_2]$, and point-wise multiplication. Its linear category of
modules $\mathrm{Rep}(\pi_2) (\pi_1)$, consisting of bundles of
$\pi_2$-representations over $\pi_1$.  Now, the tensor category of
$\bC[\pi_2][2]$-modules is equivalent to $\mathrm{Rep}(\pi_2)$ as a linear
category, but carries the non-standard tensor structure corresponding to
convolution of characters. In the Fourier transformed picture, this is the
category $\vect(\pi_2^*)$ of vector bundles on the Pontryagin dual group, but
with the \emph{point-wise} tensor structure.\footnote {The standard tensor
structure on $\mathrm{Rep}(\pi_2)$ would correspond to convolution on
$\vect(\pi_2^*)$.}  The group $\pi_1$ acts by automorphisms of the
\emph{tensor category} $\vect(\pi_2^*)$: the action comprises the obvious
automorphisms of $\pi_2$, as well as the $k$-invariant $k\in
H^3(B\pi_1;\pi_2)$, if present. Indeed, we may interpret $k$ as a crossed
homomorphism $\pi_1\to B^2\pi_2$, and $\pi_2$ is a group of central
$2$-automorphisms of $\vect(\pi_2^*)$ as a \emph{tensor category}: namely,
elements of $\pi_2$ give $1$-automorphisms (multiplications by Fourier modes)
of $\vect(\pi_2^*)$ as a bi-module category over itself, which as such
represents the identity morphism on the tensor category
$\vect(\pi_2^*)$. Thus, we can form the desired crossed-product tensor
category\footnote{An alternative view: $k$ classifies an extension of $\pi_1$
by $B\pi_2$, which twists the obvious crossed product tensor structure on
$\pi_1\ltimes\vect(\pi_2^*)$.}  $\pi_1\ltimes\vect(\pi_2^*)$ of modules for
$\pi_1 \ltimes \bC[\pi_2][2]$.  More succinctly but loosely, when $\pi_1$
acts trivially and the $k$-invariant is null, $H=\bC[\pi_2](\pi_1)$ is a Hopf
algebra over $\bC[\pi_2]$ and this gives a $2$-algebra structure as explained
in \ref{lowdim}. A non-trivial action of $\pi_1$ is incorporated by twisting
the left $H$-module structure over $H$, as is the $k$-invariant.

The case $m=3$ gives rise to $3$-algebras. Each connected component
$X_x$ is $B\Omega_x{X}$, and $R_3(X_x)$ is the promotion of 
$R_2(\Omega_x{X})$ to a $3$-algebra using the extra multiplication in 
the loop space. However, for simply connected $X$, a braided tensor 
category arises naturally. Namely, the tensor category $\pi_2\ltimes_{
\Omega{k}}\vect(\pi_3^*)$ was defined from the second loop space 
$\Omega^2{X}$, which has a homotopy-commutative multiplication, 
homotopy-commuting with the third multiplication on $R_3(X_x)$. 
(The action of $\pi_2$ is classified by the looping $\Omega{k}\in 
{H}^3(B\pi_2;\pi_3)$ of the unique Postnikov invariant of $X$, $k\in 
H^4(K(\pi_2;2);\pi_3)$.) Otherwise put, the category $\pi_2\ltimes_
{\Omega{k}}\vect(\pi_3^*)$ has a second multiplication compatible 
with the first: this structure is equivalent to a braiding. A $\pi_1$, 
of course, would spoil the requisite commutativity. 

\begin{remark}
Recall that $k$ is equivalent to the datum  of a 
quadratic map $\pi_2\to\pi_3$ (Lemma~\ref{thm:30}), so this quadratic 
map is all that is needed for the construction of the braiding on  
the category $\pi_2\ltimes_{\Omega{k}}\mathrm{Rep}(\pi_3)$. When 
$\pi_3$ is a subgroup of $\bT$, $B^2\pi_3$ acts by automorphisms of 
the $2$-algebra $\bC[2]$ (equivalently, automorphisms of the tensor 
category $\vect$), and we see from the same construction that a 
braiding on $\vect[\pi_2]$ is determined by a $\bT$-valued 
quadratic form on $\pi_2$. 
\end{remark}

\subsection{Outline of the general construction}
By now, the reader may have imagined the inductive procedure for constructing 
the $m$-algebra $R_m(X)$. The underlying vector space, or $0$-algebra, 
comprises the (finitely supported) functions on the $m$-truncated homotopy 
\[
\coprod_{x\in\pi_0X}\pi_m(X,x)\times\dots\times\pi_1(X,x). 
\]
As a $1$-algebra, we see the $\bC[\pi_m]$-valued functions on the 
union of the $(\pi_{m-1}\times\dots\times\pi_1)$, with point-wise 
multiplication. The full $m$-algebra structure can be described recursively. 
First, as an $(m-1)$-algebra, 
\[
R_m(X) = \bigoplus_{x\in\pi_0X} R_{m-1}(\Omega_x{X}).
\] 
Now, the loop spaces $\Omega_x{X}$ carry a multiplication, and the 
induced multiplication on each $R_{m-1}$ allows its promotion to an 
$m$-algebra. Finally, $R_m$ is the direct sum of the resulting $m$-algebras.

\begin{remark} Let us go one step further in unraveling this description: 
for each base-point~$x$, $\pi_1(X,x)$ acts on $B^2\Omega^2_x{X}$; this 
defines an action on the $m$-algebra $R_m(B^2\Omega^2_x{X})$, and 
we have $R_m(X_x) := \pi_1(X,x)\ltimes{R}_m(B^2\Omega^2_x{X})$.
\end{remark}

This procedure has the advantage of producing finite-dimensional 
objects whenever the homotopy groups of $X$ are finite, but the choice 
of base-points in the induction step breaks hopes of functoriality. 
(This is similar to the construction of a minimal model for $X$ from 
its Postnikov tower in rational homotopy theory.) So our notation 
$R_m(X)$ is somewhat abusive: we produce something akin to the group 
ring of the based loop space $\Omega{X}$. A remedy would be to
cross with the \emph{Poincar\'e groupoids}, instead of $\pi_1$. This 
would produce a Morita equivalent algebra; but even with a finite 
model for a higher groupoid, there is then little hope of a 
finite-dimensional answer. 

\subsection{Twisting by cohomology classes in $\bT$} \label{twistquant}
The construction of the $m$-algebra (and in fact of the entire 
$n=m+1$-dimensional TQFT) can be twisted by a class $\tau\in H^{n}(X;\bT)$ 
of our space $X$. This gives a projective co-cycle for actions of 
$\Omega{X}$ on $(m-1)$-algebras, and the twisted group ring can now again 
be defined as the $m$-algebra with the `same' representation category. 
Some examples: for $m=-1$, the points are now weighted by their phase 
in $H^0(X;\bT)$ before counting. For $m=0$, we get a non-trivial action 
of $\pi_1$ on $\bC$ for each component of $X$, and only the invariant 
lines are summed up to produce the vector space. For $m=1$, 
the class gives a central extension of each $\pi_1$, which we quantize 
to the sum of the corresponding twisted group algebras. For $m=2$ 
and connected $X$, $H^3(X;\bT)$ classifies the crossed homomorphisms  
$\pi_1\to \pi_2^*$, and $\tau$ twists the action of $\pi_1$ on 
$\vect(\pi_2^*)$, resulting in a different crossed product  
$\pi_1\ltimes_{\tau,k} \bC[\pi_2][2]$. Finally, for pure gerbes $K(\pi,2)$, 
a $4$-class $\tau$ defines a quadratic map $\pi\to\bT$, which defines 
a braided tensor structure on $\vect[\pi]$. The categories $\sky^\tau[\frt], 
\sky^\tau[\hat{F}]$ from~\S\ref{sec:c1} are of this form.  

The construction of a general twisted groupoid $m$-algebra follows the 
inductive procedure sketched earlier. Note that, in the absence of 
twistings, $R_m(X) = R_m(X)^{op}$; this is related to the fact that 
the associated TQFT does not require orientations.

\subsection{Construction of the quantization map~$\Sum_n$}  

The category $FH$ of spaces with finitely many, finite homotopy groups has
disjoint unions, products and fiber products (homotopy fiber products). As
in~\S\ref{sec:5}, define the $n$-category $FH_n$ of correspondences in $FH$
(truncated at level $n$). We will enhance the assignment $X\mapsto
R_{n-1}(X)$ into a symmetric tensor\footnote {$Sum_n$ will take disjoint
unions to direct sums and fiber products to tensor products.} functor 
  \begin{equation}\label{eq:70}
     \Sum_n: FH_n \to \mathrm{Alg}[n-1] 
  \end{equation}
into the $n$-category\footnote{This is a particular choice of~`$\C$'
in~\S\ref{sec:5}.}~$\textnormal{Alg}[n-1]$ of $(n-1)$-algebras.  Having fixed
an $X\in\mathrm{Ob}(FH)$, we pre-compose with the mapping space functor
$I:\mathrm{Bord}_n \to FH_n$ to produce the `TQFT with target $X$', a theory
with values in $\mathrm{Alg}[n-1]$.  \begin{remark} There is a similar
functor $F^\tau$ on the category $FH_n^\tau$, whose objects are spaces with
an $n$-co-cycle $\tau$ valued in $\bT$, $1$-morphisms are correspondences
equipped with a homotopy between pulled-back cocycles, and so forth. This
leads to the twisted TQFTs for \emph{oriented} manifolds; we will not spell
out its details here.
\end{remark}

Whereas on objects $\Sum_n(X)= R_{n-1}(X)$, the formula for morphisms is
increasingly complex and we pause for a moment to explain why. If $R_{n-1}$
was a \emph{contravariant} functor with ``good" tensor properties, we would
just apply it to morphisms of all levels.  The reader might think of the
algebra of co-chains on a space, which converts correspondences to
bi-modules. (As a technical aside, this does \emph{not} always have the
required tensor properties, but a closely related functor does: the
co-algebra of chains, with co-tensor products.)  However, our $R_{n-1}$ is
the Koszul dual functor of chains on the based loop space. We thus use Koszul
duality as a recipe to define $\Sum_n$.  The general rule is that a
correspondence at level $k$, given by a $k$-storied diagram of spaces, gets
sent to the colimit of the same diagram of group rings.

We illustrate this for $1$- and $2$-morphisms. We may assume connectivity 
of our spaces (separate components are handled separately). For a
correspondence $c:C \to X\times Y$, call $F$ the homotopy fiber 
of $c$. It may be viewed as the anti-diagonal quotient $(\Omega{X}\times
\Omega{Y})/\Omega{C}$; the obvious action of $\Omega{X}\times\Omega{Y}$ 
on this last space is the structural one on $F$. To $C$, we assign the 
$(R_{n-1}(X),R_{n-1}(Y))$-bi-module
 \[
\Sum_n(C): = R_{n-1}(X)\underset{R_{n-1}(C)}\bigotimes{R}_{n-1}(Y) 
	\cong R_{n-2}(F).
 \]
This is an $(n-2)$-algebra, as the top multiplication layer has been 
used up to build the tensor product. The geometric $\Omega{X}\times
\Omega{Y}$ actions on $F$ induce the bi-module structure on $R_{n-2}(F)$

Now let $D \to C\times C'$ be a correspondence of correspondences; 
the compositions $D\to X\times Y$ must agree. With $H$ denoting the 
homotopy fiber of the last map, there is induced a map $h:H \to F
\times F'$ with compatible $\Omega{X}\times\Omega{Y}$ action. The 
homotopy fiber $G$ of $h$ carries $\Omega{F}\times\Omega{F'}$ actions and 
intertwining $\Omega{X}\times\Omega{Y}$ actions; in fact, 
\begin{equation}\label{2cor}
G \sim \frac{\Omega{F}\times\Omega{F'}}{\Omega{H}}\sim \frac{\Omega^2{X} 
	\times
\Omega^2{Y}}{\Omega^2{C}\underset{\Omega^2{D}}\times\Omega^2{C'}},
\end{equation}
with the left$\times$right action of $\Omega^2{C}\times\Omega^2{C'}$ on each 
factor in the second numerator; we have just enough commutativity to 
mod out by the anti-diagonal $\Omega^2{D}$. The desired multi-module is 
 \[
\Sum_n(D) := R_{n-3}(G) = R_{n-2}(F)\underset{R_{n-2}(H)}\bigotimes R_{n-2}(F') 
= R_{n-2}(X)\underset{R_{n-2}(C)\underset{R_{n-2}(D)}\otimes{R}_{n-2}(C')}
\bigotimes{R}_{n-2}(Y);
 \]
in the first presentation, tensoring in the $(n-2)$-multiplications  
ensures that we stay in the realm of ``$(R_{n-1}(X),R_{n-1}(Y))$ bi-algebra 
$(R_{n-2}(F),R_{(n-2)}(F'))$-bi-algebras", but uses up the top two level 
products on $\Sum_n(D)$ and leaves an $(n-3)$-algebra. In the second 
presentation, we view $D$ as a $(C,C')$-correspondence in a lower-dimensional 
field theory and tensor over the $(n-2)$-algebra $\Sum_{n-1}(D)$, acting 
left$\times$right on each factor.  

The fun continues, but we will stop. 

\begin{remark} Recall the finite group~$F$ from~\ref{subsec:3.1}.
Consider the gerbe $B^2F$, equipped with the class $\tau\in
H^4(B^2F;\bT)$ described in Lemma~\ref{thm:30}.  The space $C$ of maps from a
closed $4$-manifold $X$ to $B^2F$ can be viewed as a top-level correspondence
from the point to itself, and \eqref{sumpi} is the Gauss sum computed
in~\ref{subsec:6.1}.

\end{remark}

\section{The quantum theories}

The braided tensor categories $\sky^\tau[F]$, $\sky^\tau[\frt^{(n)}]$ and
$\sky^\tau[\frt]$ generate $4$-dimensional TQFTs with target space the
corresponding ($\tau$-twisted) gerbes. Their $4$-manifold invariants were
computed as a path integral in~\ref{subsec:6.1}; we have seen their agreement
with the numbers provided by the general quantization procedure. The
invariants for $\frt$ need to be renormalized by a power of the volume of
$\frt$; this is explained by the increasing number of points in its finite
approximation $\frt^{(n)}$.  We discuss these theories further
in~\ref{cs9.2}.

\subsection{The quantum gerbe theories on $F, \frt$ and
$\frt^{(n)}$}\label{cs9.2}  

We now describe the theories $\cA_S$ on closed manifolds of dimension 
below $4$; $\cA_\frt$ can be justified either as the limit of $\cA_{
\frt^{(n)}}$'s, or in its own right by judicious use of the word 
`continuous.' The Lie algebra theory also has a few interesting variations, 
which tie in with positive energy representations of $L\frt$ and 
their fusion. 
\vskip2ex
\noindent\emph{The point.} We already know the braided categories 
$\sky^\tau[\hat{F}]$,  $\sky^\tau[\frt]$, and~$\sky^\tau [\frt^{(n)}]$
assigned by $\cA$ to a  
positively oriented point;  $\cA$ sends the negative point to the 
opposite object, the category with opposite braiding. Recall that, 
in all cases, the twisting $\tau\in H^4(-;\bT)$ which defines the 
braided structure is given by a non-degenerate quadratic map $q$ 
to $\bT$ (see \eqref{eq:49}, \eqref{eq:55}, \ref{cs6.1}).
\vskip2ex

\noindent\emph{The circle.} Here, $\cA_F$ assigns the $2$-algebra
quantization of the groupoid $LB^2F = BF\times B^2F$, with twisting
$\Omega\tau\in H^3(LB^2F;\bT)$ transgressed from $\tau$; this class~$\Omega
\tau $ represents the bihomomorphism $b: F\times{F} \to\bT$ derived from
$q$. According to \ref{twistquant}, the quantization is the crossed product
$R : = F\ltimes_{\Omega\tau}\bC[F][2]$ with action twisted by the
transgression of $\tau$. We have several pictures for this.

\begin{enumerate}
\item The tensor category of $R$-modules is $F\ltimes_{b}\vect(F^*)$, 
after identifying $\bC[F]$-modules with vector bundles on $F^*$. The 
group $F$ acts on $F^*$ by translation, via $b$. Note that this is 
equivalent to the matrix algebra $M_{F^*\times{F}^*}(\vect)$ on 
$\vect(F^*)$! 
\item $R$-modules in algebras are $\bC[F]$-algebras with an action 
of the group $F$ by automorphisms, which twisted commutes with the 
central $\bC[F]$: $a_f(f') = f' \cdot{b}(f,f')$ for $f,f'\in F$, 
with $a_f$ denoting the 
automorphism and $f'\in F$ embedded in the central $\bC[F]$. 
\item Related to this is the $2$-category of $R$-linear categories: 
these are $\bC$-linear categories with a projective action of $F\times BF$ 
with co-cycle $b$. This is an action of $F$ by linear functors plus a 
second action of $F$ by automorphisms of the identity functor (central
automorphisms of all objects), which must be related by $a_f(f'_x) = 
f'_{a_f(x)} \cdot{b}(f,f')$, for each object $x$. 
\item The same $2$-category has a different presentation which emphasizes 
its loopy nature: the $2$-category $\cB{r}^\tau(F)$ of (fully) 
\emph{braided bi-module categories} over the braided category $\sky^\tau[F]$. 
The full braiding gives the $BF$-action. We will see that its counterpart 
$\cB{r}^\tau(\frt)$ relates nicely to positive energy representations 
of $L\frt$. 
\end{enumerate}

The algebra $\bC[F]$ is a module object over $R$: this is because the 
linear category $\mathrm{Rep}(F)$ of $\bC[F]$-modules is naturally a 
module over the tensor category $F\ltimes_{b}\vect(F^*)$ in (i); more 
precisely, it is the standard simple module of $M_{F^*\times{}F^*}(\vect)$. 
This implies the following.

\begin{proposition}
For non-degenerate $q$, $R$ is Morita equivalent 
to $\bC[2]$, with bi-module $\bC[F]$. \qed
\end{proposition}
 
\noindent\emph{Closed Surfaces.} The space of maps from a closed 
surface $\Sigma$ to $B^2F$ factors as $B^2F\times{}H^1(\Sigma;F) \times
F$. The transgression $\tau_\Sigma$ of the twisting co-cycle has two 
components: one of them defines the Heisenberg central extension of 
$H^1(\Sigma;F)$ constructed from $q$ and the Poincar\'e duality 
pairing, the other, on $B^2F$, gives the character $f\mapsto b(f,f_0)$, 
in the component labeled by $f_0$. Quantization produces the group 
ring, the Weyl algebra $W_\tau\big(H^1(\Sigma;F)\big)$ for the 
component $f_0=0$, and kills the other components. 

At this stage, invertibility of the theory has become obvious, because 
we get a matrix algebra. Despite this, the theory is \emph{not} trivial 
on surfaces. Indeed, while the mapping class group of $\Sigma$ acts by 
automorphisms on the Weyl algebra, this conceals a central extension by 
eighth roots of unity, given (when $\frt$ has rank one) by the reduction 
mod $8$ of the Maslov index. This appears when attempting to lift the 
action on Lagrangian subspaces in $H^1$ to the associated Schr\"odinger 
representations~\cite[\S I.4]{P}.

\begin{remark}
Even when this extension is trivial, such as for spin surfaces, 
two trivializations may differ by a (half-integral) power of the 
determinant line. This stems from an \emph{invertible} $3$-dimensional 
theory, and the usual framing anomaly in Chern-Simons theory can 
be concealed therein. 
\end{remark} 
 
\vskip1ex

\noindent\emph{$3$-folds with boundary.}
A bounding $3$-manifold $M$ gives a Lagrangian subspace in  $H_1(\Sigma)$;  
we can then form the Schr\"odinger representation $\mathbf{S}$ of 
$W(\partial{M})$~\cite{P}. This is the morphism $\cA_F(M)$ 
from $\bC$ to the Weyl algebra. Gluing two such manifolds into a closed one 
produces the line $\mathrm{Hom}_{W(\partial{M})}(\mathbf{S_-},\mathbf{S}_+)$. 
This line has a preferred trivialization on any oriented $3$-manifolds, 
but also carries a natural action of the group $\bZ$ of global frame changes. 

\subsection{Chern-Simons theory}\label{sec:c9}

Consider the $3$-dimensional gauge theory $Z_{(n)}$ with finite group
$T_{(nF)}$ and twisting class $\tau\in H^3(BT_{(nF)};\bT)$. Such theories
were described in ~\ref{subsec:1.5}.  From the perspective of
Theorem~\ref{bd}, the theory~$Z_n$ is generated by the tensor
category\footnote{In $2$-algebra language, $\sky^\tau[T_{(nF)}]$ is the
tensor category of modules for the algebra of functions on $T_{(nF)}$, with
$\tau$-twisted Hopf $2$-algebra structure.}
$\sky^\tau[T_{(nF)}]$. Similarly, we should view $\sky^\tau[T]$ as the
generating tensor category for a theory $Z$, which is an $L^2$ version of
Chern-Simons theory with gauge group $T$. For example, the vector space
associated to a closed surface is the space of $L^2$ sections of the
Theta-line bundle~$\Theta (\tau )$ on the $T$-Jacobian $J_T$ of flat bundles
(whereas the usual, holomorphic Chern-Simons theory would supply the
holomorphic sections). It is easiest to justify the relation of
$\sky^\tau[T]$ to $L^2$ gauge theory by using the finite approximations
$Z_{(n)}$: sections of $\Theta(\tau)$ on $T_{(nF)}$-bundles approximate
$L^2(J_T)$ as $n\to\infty$.  We will see in~\ref{cs9.2} that this goes for
$1$- and $3$-dimensional outputs as well. (The limit must be regularized, so
that, for instance, summation over torsion points becomes integration in the
limit. $L^2$-Chern-Simons itself requires regularization: vectors of bounding
$3$-manifolds are $\delta$-sections on the respective Lagrangians in $J_T$,
and the $3$-manifold invariant is only 'obviously' finite for for rational
homology spheres.)

One picture of a $3$-dimensional TQFT, as a stand-alone theory, is as an
endomorphism of the trivial $4$-dimensional theory, with the top level
truncated.\footnote{Without truncation, all morphisms between TQFTs are
isomorphisms. At any rate, a $3D$ TQFT does not supply anything in dimension
$4$.} If, however, a $3D$ theory is a module over another, non-trivial $4d$
TQFT $\cA$---meaning that the algebras arising various dimensions is are
module objects over their $4d$ counterparts, in a consistent way---then it
can be viewed as a truncated morphism $1\to \cA$. If $\cA$~is invertible,
this is the notion of an `anomalous field theory' with anomaly~$\cA$.
Similarly, a bi-module for theories $(\cA_\frt, {\cA}_F)$ yields a truncated
morphism ${\cA}_F \to \cA_\frt$. Of course, concerning $Z$, we have the gerbe
theories for $F$ and $\frt$ in mind.  The following summarizes our results.

\begin{theorem}\label{thm:9.1}
\begin{trivlist}\itemsep0ex
\item (i) The $L^2$-Chern-Simons theory $Z$ generated by $\sky^\tau[T]$ 
gives an \emph{isomorphism} $\cA_F \to \cA_\frt$ of oriented 
$4d$ theories.
\item (ii) The finite gauge theory $Z_{(n)}$ generated by $\sky^\tau
[T_{(nF)}]$ gives an \emph{isomorphism} ${\cA_F} \to \cA_{\frt^{(n)}}$ 
of oriented $4d$ theories. 
\item (iii) $Z$ gives a \emph{truncated morphism} $1 \to \cA_\frt$ of 
oriented $4d$ theories. 
\item (iv) A Morita equivalence $\vect\sim \sky^\tau[\frt]$ induces 
an isomorphism $\cA_\frt \to 1$; after a composition with such 
an isomorphism, $Z$ becomes isomorphic to $3d$ (holomorphic) 
Chern-Simons theory.
\item (v) $Z_{(n)}$ gives a \emph{truncated morphism} $1 \to 
\cA_{\frt^{(n)}}$ of oriented $4d$ theories. A Morita equivalence 
$\vect\sim \sky^\tau[\frt^{(n)}]$ induces an isomorphism 
$\cA_{\frt^{(n)}} \to 1$; after a composition with such 
an isomorphism, $Z_{(n)}$ becomes isomorphic to $3d$ (holomorphic) 
Chern-Simons theory.
\end{trivlist}
\end{theorem}

\begin{remark}
 Recall from~\ref{cs6.3} that a Morita isomorphism as in (iii) and (v) can be
found whenever $8$ divides the signature of $\tau$, or anytime we work with
graded objects and spin manifolds. Items (ii) and (v) are perhaps not of much
interest, but in the $n\to\infty$ limit they serve to justify rigorously the
claims (i), (iii) and (iv); without that, we would need to delve into
topological categories and their continuous centers.  The truncated
morphism~$Z$ in~(iii) is (holomorphic) Chern-Simons as an anomalous theory.
\end{remark}

\noindent The theorem is really a corollary of Proposition~\ref{morita}---or
at least it would be so, if we supplied the information needed to make the
braided tensor categories into fully dualizable objects with $SO(4)$-actions
in the world of braided tensor categories.\footnote {Something is missing, as
can be seen from the invariant computation in~\ref{subsec:6.1}, where an
Euler characteristic-coupled parameter must be supplied.}  We shall not do
this; instead let us explain what the theories assign in each dimension. In
particular, we will recover the usual modular tensor category~\cite{St} for
holomorphic Chern-Simons for $T$ as the relative Drinfeld center of
$\sky^\tau[T]$ over $\sky^\tau[\frt]$.

\subsection{Chern-Simons as an anomalous theory}\label{subsec:c9.3}  

We condense the relations between our theories on closed manifolds $X$ in the
following table, in which the third column gives a Morita isomorphism between
the second and fourth columns. We have written $Z(X)$ for the $L^2$
Chern-Simons invariant of a $3$-manifold, a renormalized count of the flat
$T$-bundles on $X$~\cite{Ma}, while $\cA_{F,\frt}(X)$ refers to the
$4$-manifold invariants computed in~\ref{subsec:6.1}.

\bigskip

\begin{center} \begin{tabular}{ c c c c }  
  \toprule 
  $\dim{X}$  & $\cA_\frt(X)$ & $Z(X)$  & $\cA_F(X)$\\  
  \midrule 
  $0$ & $\sky^\tau[\frt]$ & $\sky^\tau[T]$ & $\sky^\tau[\hat{F}]$ \ \\ [-8pt]\\  
  $1$ & $\frt\ltimes_b\sky[\frt^*]$ & $\sky^\tau[C]$ & $F\ltimes_b\sky[F^*])$ 
  							\\ [-8pt]\\
  $2$ & $W_\tau\big(H^1(X;\frt)\big)$ &$L^2\big(J_T(X); \Theta(\tau)\big)$& 
  			$W_\tau\big(H^1(X;F)\big)$\\  [-8pt]\\
  $3$ & $\bC$ & $Z(X)$ & $\bC$  \\ [-8pt]\\
  $4$ & $\cA_\frt(X)$ & --- & $\cA_F(X)$ \\
     \bottomrule \end{tabular} \vskip2ex
     \begin{small}{Theories on closed manifolds}\end{small} \end{center}

\bigskip
There is a matching equivalence with the finite theories for 
$T_{(n)}, \frt^{(n)}$:
\bigskip

\begin{center} \begin{tabular}{ c c c c }  
  \toprule 
  $\dim{X}$  & $\cA_{\frt^{(n)}}(X)$ & $Z_{(n)}(X)$  & $\cA_F(X)$\\  
  \midrule 
  $0$ & $\sky^\tau[\frt^{(n)}]$ & $\sky^\tau[T_{(nF)}]$ & $\sky^\tau[\hat{F}]$ 
  \\[-8pt]\\  
  $1$ & $\frt^{(n)}\ltimes_b\sky[\frt^{*(n)}]$ & $\sky^\tau[C^{(n)}]$ &
$F\ltimes_b\sky[F^*])$ \\ [-8pt]\\
  $2$ & $W_\tau\big(H^1(X;\frt^{(n)})\big)$ &$L^2\big(J_{T_{(nF)}}(X); 
  			\Theta(\tau)\big)$ & $W_\tau(H^1(X;F))$\\  [-8pt]\\
  $3$ & $\bC$ & $Z_{(n)}(X)$ & $\bC$  \\ [-8pt]\\
  $4$ & $\cA_{\frt^{(n)}}(X)$ & --- & $\cA_F(X)$ \\
     \bottomrule \end{tabular} \vskip2ex
     \begin{small}{Finite theories on closed manifolds}\end{small} \end{center}

\bigskip
\noindent Experts may have recognized in the right column the \emph{double} 
of Chern-Simons theory for the torus $T$ at level $\tau$. This is explained 
as follows. 

The second and third columns in each table give our advertised description of 
(holomorphic) Chern-Simons theory as an anomalous theory. We claim that 
$Z$ is `finite as a module over $\cA_\frt$, with holomorphic Chern-Simons 
theory as a basis' (and similarly for $Z_{(n)}$ and $\cA_{\frt^{(n)}}$).
The most obvious instance is that of the vector space 
\begin{equation}\label{L2funcs}
Z(\Sigma) = L^2\big(J_T(\Sigma); \Theta(\tau)\big),
\end{equation}
for a closed surface $\Sigma$. The theory of Theta-functions tells us that,
after a choice of complex structure, \eqref{L2funcs} factors into 
holomorphic and anti-holomorphic sectors; the former is the space of 
holomorphic Theta-functions, an irreducible representation of the 
Heisenberg group on $H^(\Sigma;F)$, and the latter the anti-holomorphic 
Fock representation of the Weyl algebra $W\big(H^1(\Sigma;\frt)\big)$. 

One dimension down, the center $\sky^\tau[C]$ of $\sky^\tau[T]$, plays the
role of the modular tensor category for $L^2$ Chern-Simons.  The braided
structure on $\sky^\tau[C]$ makes this into a braided bi-module over
$\sky^\tau[\frt]$, thus a $\frt\ltimes_b\sky[\frt^*]$-module. It is free
with basis $\sky^\tau[\hat{F}]$, in the sense that it converts into the
latter, after the Morita equivalence of $\frt \ltimes_b\sky^\tau[\frt^*]$
with $1$ defined by $\sky[\frt^*]$. Such a Morita equivalence can be induced
from a trivialization of $\cA_\frt$ by means of $\sky^{\tau'}[T]$ at a
level $\tau'$ giving a trivial group $F'$;\footnote{Using graded vector
spaces, see~\ref{cs6.3}.} for indeed, $\sky^{\tau'}[C']$ is then just
$\sky[\frt^*]$, as a braided bi-module.  (The change to another level of the
same signature can be accommodated by scaling the Lie algebra.) This is also
related to the computation of the relative center of $\sky^\tau[T]$ over the
braided category $\sky^\tau[\frt]$.

In dimension $0$, our data for Chern-Simons theory is new.

\subsection{Surfaces with boundary in another model for $\frt$-gerbe theory}
In this model $\mathrm{Rep}(L\frt)$ for $\sky^\tau[\frt]$, closely related 
to loop groups, the objects are semi-simple, projective, positive energy 
modules of the smooth loop Lie algebra $L\frt$, with projective co-cycle
 \[
(\xi,\eta) \mapsto \oint \;\langle\xi,d\eta\rangle.
 \]
These representations are invariant under diffeomorphisms of the circle; 
non-invariantly, they factor into  a semi-simple representation of $\frt$ 
and the Fock representation of $L\frt/\frt$, the latter being a Morita 
factor in the equivalence.

The \emph{fusion} of two representations is defined using a pair 
of pants $P$: the Weyl algebra $W_\tau(P)$ of the pair of pants 
(see below) accepts maps from three commuting copies of the Weyl 
algebra $W_\tau(L\frt)$ of $L\frt$. The fusion of two boundary 
representations is the induced module from their tensor product, 
as a $W_\tau(P)$-module, to the the third boundary. Fusion gives 
a braided tensor structure on $\mathrm{Rep}(L\frt)$, which makes it 
equivalent to $\sky^\tau[\frt]$. $W_\tau(L\frt)$ itself is the underlying 
$3$-algebra, when equipped with the fusion product via the braided 
tri-module $W_\tau(P)$.

Associated to a surface $\Sigma$ with boundary is the symplectic vector 
space $S_\Sigma:= \Omega^1(\Sigma)/d\Omega^0(\Sigma,\partial\Sigma)$ of 
closed forms modulo differentials of functions vanishing on $\partial\Sigma$; 
the symplectic form is $\int_\Sigma \varphi\wedge\psi$. The Weyl 
algebra $W_\tau(S_\Sigma\otimes\frt)$ is a braided bi-algebra for the 
$W_\tau(L\frt)$'s at each boundary circle. This promotes $W_\tau(\Sigma)$ 
to an object of $\cB{r}^\tau$, the $2$-category (equivalent to that) of 
$\frt\ltimes_b\sky[\frt^*]$-modules, and the thus promoted $W_\tau(\Sigma)$
is $\cA_\frt(\Sigma)$. 

A model $S'_\Sigma $ for $S_\Sigma $ that removes the Morita factors
$\mathrm{Fock} (L\frt/\frt)$ from $W$ uses only differentials whose
restriction to each boundary circle is constant (in a fixed
parameterization). We lose the connection to loop groups, but this is
convenient for describing $Z(\Sigma)$ (without having to Morita-modify
$C$). Now, $Z(\Sigma)$ is a functor between products of copies of
$\sky^\tau[C]$, one for each boundary circle: in effect, a vector bundle over
a product $C^n$.  The moduli space $J_T(\Sigma,\partial\Sigma)$ of flat
$T$-bundles on $\Sigma$ equipped with \emph{constant} connections on
$\partial\Sigma$ projects to $\frt^n$ by the boundary holonomies; call this
map $p$.  $Z(\Sigma)$ is the bundle over $\frt^n$ of fiber-wise $L^2$
sections of $\Theta(\tau)$ along $p$; it is naturally a finite module over
$W_\tau(S'\otimes\frt)$ under its translation of the Jacobian. The
$\hat{F}$-components of $Z(\Sigma)$ are determined by the weight space
decomposition under the actions of $F$ at the boundaries.


\begin{thebibliography}{M{\"u}g03}

\bibitem[A]{A}
M.~Atiyah, \emph{Topological quantum field theories}, Inst. Hautes \'Etudes
  Sci. Publ. Math. (1988), no.~68, 175--186 (1989).

\bibitem[AS]{AS}
M.~Atiyah, G.~Segal, \emph{Twisted {$K$}-theory}, Ukr. Mat. Visn. \textbf{1}
  (2004), no.~3, 287--330.

\bibitem[BD]{BD}
J.~C. Baez, J.~Dolan, \emph{Higher-dimensional algebra and topological
  quantum field theory}, J. Math. Phys. \textbf{36} (1995), no.~11, 6073--6105.

\bibitem[BK]{BK}
 B. Bakalov, A. Kirillov, Jr., \emph{Lectures on tensor categories
  and modular functors}, University Lecture Series, vol.~21, American
  Mathematical Society, Providence, RI, 2001.

\bibitem[BW]{BW} 
J. W. Barrett, B. W. Westbury, \emph{Spherical categories}, Adv. Math., {\bf
143} (1999), 357--375, {\tt arXiv:hep-th/9310164}.

\bibitem[BM]{BM} 
 D. Belov, G. W. Moore, \emph{Classification of abelian spin Chern-Simons
theories,}, {\tt  arXiv:hep-th/0505235}.

\bibitem[BN]{BN} 
 D. Ben-Zvi, D. Nadler, \emph{The character theory of a complex group}, {\tt
arXiv:0904.1247}. 

\bibitem[B]{B}
R.~Bott, \emph{Lectures on characteristic classes and foliations}, Lectures on
  algebraic and differential topology ({S}econd {L}atin {A}merican {S}chool in
  {M}ath., {M}exico {C}ity, 1971), Springer, Berlin, 1972, Notes by Lawrence
  Conlon, with two appendices by J. Stasheff, pp.~1--94. Lecture Notes in
  Math., Vol. 279.

\bibitem[Bry]{Bry}
J-L Brylinski, \emph{Gerbes on complex reductive Lie groups}, {\tt
  arXiv:math/0002158}.

\bibitem[CY]{CY}
L.~Crane, D.~Yetter, \emph{A categorical construction of {$4$}{D}
  topological quantum field theories}, Quantum topology, Ser. Knots Everything,
  vol.~3, World Sci. Publ., River Edge, NJ, 1993, pp.~120--130.

\bibitem[D]{D} 
C. Douglas, \emph{Two-dimensional algebra and quantum Chern-Simons theory},
lecture at Topological Field Theory conference, Northwestern University, May,
2009. 

\bibitem[EL]{EL}
 S.~Eilenberg, S.~Mac Lane, \emph{On the groups $H(\Pi,n)$. II. Methods of
computation}, Ann. of Math. \textbf{60} (1954), 49--139.

\bibitem[FHT]{FHT}
C.~Teleman D.~S.~Freed, M. J.~Hopkins, \emph{Loop groups and twisted $K$-theory
  I, II, III}, {\tt arXiv:0711.1906, arXiv:math/0511232, arXiv:math/0312155}.

\bibitem[F1]{F1}
D.~S. Freed, \emph{Classical {C}hern-{S}imons theory. {II}}, Houston J. Math.
  \textbf{28} (2002), no.~2, 293--310, Special issue for S. S. Chern.

\bibitem[F2]{F2}
\bysame, \emph{Higher algebraic structures and quantization}, Comm. Math.
  Phys. \textbf{159} (1994), no.~2, 343--398.

\bibitem[HS]{HS}
M.~J. Hopkins, I.~M. Singer, \emph{Quadratic functions in geometry,
  topology, and {M}-theory}, J. Differential Geom. \textbf{70} (2005), no.~3,
  329--452.

\bibitem[HM]{HM}
 D.~Husemoller, J.~Milnor, \emph{Symmetric bilinear forms}, Springer-Verlag,
  New York, 1973, Ergebnisse der Mathematik und ihrer Grenzgebiete, Band 73.

\bibitem[L]{L} J.~Lurie, \emph{On the classification of topological field
theories}, available at {\tt
  http://www-math.mit.edu/$\sim$lurie/}.

\bibitem[Ma]{Ma}  
 M. Manoliu, \emph{Abelian Chern-Simons theory. I. A topological quantum
field theory}, J. Math. Phys. 39 (1998), no. 1, 170--206.
 \emph{II. A functional integral approach}, J. Math. Phys. 39 (1998), no. 1,
207--217. 

\bibitem[M]{M}
M.~M{\"u}ger, \emph{From subfactors to categories and topology. {II}. {T}he
  quantum double of tensor categories and subfactors}, J. Pure Appl. Algebra
  \textbf{180} (2003), no.~1-2, 159--219.

\bibitem[P]{P} 
 A. Polishchuk, \emph{Abelian varieties, theta functions and the Fourier
transform}, Cambridge Tracts in Mathematics, 153, Cambridge University Press,
Cambridge,  2003.   

\bibitem[Q]{Q} 
 F. Quinn, \emph{Lectures on axiomatic topological quantum field theory}, in
Geometry and Quantum Field Theory, IAS/Park City Mathematics Series,
Volume~1, D. S. Freed and K. K. Uhlenbeck, eds., American Mathematical
Society, Providence, RI, 1995, pp.~323--453.

\bibitem[S]{S}
G.~B. Segal, \emph{Stanford Lectures, Lecture 1: Topological field theories},
 available at {\tt http://www.cgtp.duke.edu/ITP99/segal/stanford/lect1.pdf}.

\bibitem[St]{St} 
 S. D. Stirling, \emph{Abelian Chern-Simons theory with toral gauge group,
modular tensor categories, and group categories}, {\tt arXiv:0807.2857}.


\bibitem[T]{T}
V.~G. Turaev, \emph{Quantum invariants of knots and 3-manifolds}, de Gruyter
  Studies in Mathematics, vol.~18, Walter de Gruyter \& Co., Berlin, 1994.

\bibitem[W]{W}
K.~Walker, \emph{Fields, blobs, and TQFTs}, available at {\tt
  http://canyon23.net/math/talks/}.

\bibitem[Wi]{Wi}
E.~Witten, \emph{Quantum field theory and the {J}ones polynomial}, Comm. Math.
  Phys. \textbf{121} (1989), no.~3, 351--399.

\end{thebibliography}

\providecommand{\bysame}{\leavevmode\hbox to3em{\hrulefill}\thinspace}
\providecommand{\MR}{\relax\ifhmode\unskip\space\fi MR }
\providecommand{\MRhref}[2]{%
  \href{http://www.ams.org/mathscinet-getitem?mr=#1}{#2}
}
\providecommand{\href}[2]{#2}

\end{document}